\documentclass[]{interact}

\usepackage{fix-cm}
\usepackage[utf8]{inputenc}
\usepackage{graphicx}
\usepackage{subfigure}
\usepackage{enumitem}
\usepackage{xfrac}
\usepackage{paracol}
\usepackage{pifont}
\usepackage{tikz}
\usepackage{calc}
\usepackage{parcolumns}
\usepackage{diagbox}
\usepackage{hyperref}
\usepackage[all]{nowidow}
\usepackage[natbibapa,nodoi]{apacite}
\newcommand{\cmark}{\ding{51}}%
\newcommand{\xmark}{\ding{55}}%

\DeclareUnicodeCharacter{0080}{ }

\UseRawInputEncoding

\theoremstyle{plain}
\newtheorem{Theorem}{Theorem}[section]
\newtheorem{Proposition}[Theorem]{Proposition}
\newtheorem{Lemma}[Theorem]{Lemma}
\theoremstyle{definition}
\newtheorem{Definition}[Theorem]{Definition}
\newtheorem{Fact}[Theorem]{Fact}

\newtheorem{Corollary}[Theorem]{Corollary}

\begin{document}
\sloppy

\title{Split Interpolation\\
\small Refining Craig's Theorem via Three-Valued Logics}

\author{
\name{Quentin Blomet\textsuperscript{a,b}\thanks{CONTACT Quentin Blomet. Email: quentin.blomet@ens.psl.eu}}
\affil{\textsuperscript{a}Institut Jean-Nicod (CNRS, ENS-PSL, EHESS), PSL University, 29 rue d'Ulm, Paris, 75005, France; 
\textsuperscript{b}Department of Philosophy, University of Greifswald
Baderstra{\ss}e 6, Greifswald, 17489, Germany}
}

\maketitle

\begin{abstract}
Which choices of truth tables and consequence relations for two logics $\mathsf{L}_1$ and $\mathsf{L}_2$ ensure the satisfaction of the following \textit{split interpolation} property: If two formulas $\phi$ and $\psi$ share at least one propositional atom and $\phi$ classically entails $\psi$, then there is a formula $\chi$ that shares all its propositional atoms with both $\phi$ and $\psi$, such that $\phi$ entails $\chi$ in $\mathsf{L}_1$ and $\chi$ entails $\psi$ in $\mathsf{L}_2$? We identify the cases in which this property holds for any pair of propositional logics based on the same three-valued Boolean normal monotonic scheme for connectives and two monotonic consequence relations. Since the resulting logics are subclassical, every instance of this property constitutes a particular refinement of Craig's deductive interpolation theorem, as it entails the latter and further restricts the range of possible interpolants.
\end{abstract}


\section{Introduction}
\cite{Craig1957}'s deductive interpolation theorem is a foundational result of classical logic that can be formulated along these lines:

\vspace{7pt}
Given any two formulas $\phi$ and $\psi$ built from the classical propositional language, if $\phi$ and $\psi$ share some atom and $\phi$ classically entails $\psi$, then there is a formula $\chi$ such that each of its atoms is in $\phi$ and $\psi$, and $\phi$ classically entails $\chi$ and $\chi$ classically entails $\psi$.
\vspace{7pt}

\noindent
Since 1957, several interpolation results have been established for other logics: by \cite{Schutte1962} for first-order intuitionistic logic, by \cite{Maksimova1977} for superintuitionistic logics, or by \cite{NolaLettieri2000} for axiomatic extensions of {\L}ukasiewicz logic, just to name a few. Interpolation results have been applied in a wide range of domains, from the study of explanation \citep{Hintikka1999} to SAT-based model checking \citep{McMillan2003}.

In recent articles, \cite{milne2016refinement} and \cite{Prenosil2017} presented a refined version of the deductive interpolation theorem, resorting to the three-valued logics $\mathsf{K}_3$ and $\mathsf{LP}$. The Strong Kleene logic $\mathsf{K}_3$ is a paracomplete logic developed by \cite{Kleene1952-KLEITM} to account for partial functions, while \cite{Asenjo1966} and \cite{Priest1979}'s Logic of Paradox $\mathsf{LP}$ is a paraconsistent logic developed to account for paradoxical statements. Milne's \textit{split interpolation} theorem goes as follows:

\vspace{7pt}
For any two formulas $\phi$ and $\psi$, if $\phi$ is satisfiable, $\psi$ falsifiable, and $\phi$ classically entails $\psi$, then there is a formula $\chi$ such that each of its atoms is in $\phi$ and $\psi$, and $\phi$ entails $\chi$ in $\mathsf{K}_3$ and $\chi$ entails $\psi$ in $\mathsf{LP}$.\footnote{
Note the difference in the antecedent of the conditional in Craig's interpolation theorem compared to Milne's. While Craig's theorem requires that $\phi$ and $\psi$ share an atom, Milne's requires that $\phi$ is satisfiable and $\psi$ falsifiable. The former condition is strictly weaker than the latter in the absence of constants in the language. If $\phi$ is satisfiable, $\psi$ falsifiable, and $\phi$ classically entails $\psi$, then $\phi$ and $\psi$ must necessarily share some atom. For the remainder of this paper, we will focus exclusively on the former condition, as any interpolation result stated with this condition is strictly stronger than those stated with the latter.}
\vspace{7pt}

\noindent
The refinement of Craig's deductive interpolation theorem is evident from the fact that classical logic extends both $\mathsf{K}_3$ and $\mathsf{LP}$. Specifically, if $\phi$ entails $\psi$ in $\mathsf{K}_3$ or $\mathsf{LP}$, then $\phi$ entails $\psi$ in classical logic. Therefore, Craig's theorem directly follows from Milne's refinement. 

The operations of both $\mathsf{K}_3$ and $\mathsf{LP}$ are defined over the set $ \lbrace 0, \sfrac{1}{2}, 1 \rbrace$ using the \textit{Strong Kleene} truth tables. Negation is characterized as an order-reversing, involutive unary operation, while disjunction and conjunction are binary operations corresponding to the functions $\mathit{max}$ and $\mathit{min}$ over the total order $0 < \sfrac{1}{2} < 1$. The distinction between $\mathsf{K}_3$ and $\mathsf{LP}$ lies in the preservation of different designated values from premises to conclusions in valid inferences. In $\mathsf{LP}$, if the premises take values in the set $\mathbf{t} = \lbrace \sfrac{1}{2}, 1 \rbrace$, then the conclusion does as well. In contrast, $\mathsf{K}_3$ preserves the truth values within the set $\mathbf{s} =  \lbrace 1 \rbrace$. The Strong Kleene truth tables belong to a class of truth tables referred to as \textit{Boolean normal monotonic schemes} by \cite{da2023three}. Informally, a scheme for the connectives is \textit{Boolean normal} if all its operations coincide with the classical two-valued ones when applied to classical values. A scheme is \textit{monotonic} if all its operations are monotonic with respect to the information order $\le_{I}$, where $\sfrac{1}{2} <_I 1$ and $\sfrac{1}{2} <_I 0$. Alongside the two standards of evaluation $\mathbf{ss}$ (which ensures the preservation of $1$ from premises to conclusion) and $\mathbf{tt}$ (which ensures the preservation of both $1$ and $\sfrac{1}{2}$ from premises to conclusion), Da R{\'e} et al.\ examine the standards $\mathbf{st}$, $\mathbf{ts}$, and $\mathbf{ss} \cap \mathbf{tt}$. The $\mathbf{st}$ standard is defined so that if the premises take a value in $\mathbf{s}$, the conclusion must take a value in $\mathbf{t}$, and conversely for $\mathbf{ts}$. The $\mathbf{ss} \cap \mathbf{tt}$ standard, on the other hand, requires that an argument be valid under both the $\mathbf{ss}$ and $\mathbf{tt}$ conditions. All five standards correspond to the \textit{intersective mixed} consequence relations over the set $\lbrace 0, \sfrac{1}{2}, 1 \rbrace$, which naturally generate the class of three-valued \textit{monotonic} consequence relations (see \citealp{Chemlaetal2017}).\footnote{This notion of monotonicity should not be confused with the one used to characterize schemes. It describes consequence relations such that for all set of formulas $\Gamma, \Delta$, if $\Gamma$ entails $\Delta$, then $\Gamma \cup \Gamma'$ entails $\Delta \cup \Delta'$, for any set of formulas $\Gamma', \Delta'$.} Da  R{\'e} et al.\ show in particular that when combined with a Boolean normal monotonic scheme, the $\mathbf{st}$ standard yields a logic in which the valid inferences align precisely with those of classical logic. 

These results provide an opportunity to extend Milne's interpolation theorem and address an open problem posed by P\v{r}enosil, which seeks additional natural examples of the split interpolation property.\footnote{\url{https://sites.google.com/site/adamprenosil/problems}} To achieve this, we will explore the following question: Which choice of three-valued logics $\mathsf{L}_1$ and $\mathsf{L}_2$ based on a Boolean normal monotonic scheme and a monotonic consequence relation guarantees that the following property holds true:

\vspace{7pt}
If two formulas $\phi$ and $\psi$ share some atom, and $\phi$ classically entails $\psi$, then there is a formula $\chi$ such that each of its atoms is in $\phi$ and $\psi$, and $\phi$ entails $\chi$ in the logic $\mathsf{L}_1$ and $\chi$ entails $\psi$ in the logic $\mathsf{L}_2$.
\vspace{7pt}

Out of the 400 possible combinations of the 25 couples of standards with the 16 Boolean normal monotonic schemes, we will see that 40 satisfy the split interpolation property, with 8 of these offering stronger refinements of Craig's interpolation theorem than the others.

To prove our claim, we start by defining the key notions used throughout the paper in Section \ref{sec:def}, in particular the notions of monotonicity and Boolean normality. In Section \ref{sec:nri}, we answer our main question for pairs of logics for which the property fails, independently of the choice of connective scheme. Section \ref{sec:rdc} is dedicated to the results dependent on a choice of scheme. It is divided into two parts, the first is dedicated to the positive results and the second to the negative ones. We finally conclude by discussing our results in Section \ref{sec:dis} and comparing them to those of Milne and P\v{r}enosil.

\section{Preliminary definitions}\label{sec:def}
In the following section, we define the terminology used throughout the paper and present the semantics of each logic discussed. The semantics will be presented through two key components: a (mixed) consequence relation that establishes the criteria for evaluating premises and conclusions, and a set of truth tables for the connectives, referred to as a \textit{scheme} for the logic \citep{da2023three}.

Consider a language composed of an enumerable set $\mathit{Var}$ of atomic formulas $p_1, \dots, p_n$ and a collection $\lbrace \neg, \wedge, \lor \rbrace$ of logical constants. The set of formulas $\mathcal{L}$ is recursively defined from the language in the usual way. Elements of $\mathcal{L}$ will be denoted by lower case Greek letters. Subsets of $\mathcal{L}$ will be denoted by uppercase Greek letters. An inference on $\mathcal{L}$ is an ordered pair $\langle \Gamma, \Delta \rangle$ (hereafter symbolized by $\Gamma \Rightarrow \Delta$). Given a formula $\phi$, the set of atoms from which it is constructed is denoted by $\mathit{At}(\phi)$.

As mentioned above, the semantics of the logics is presented through two key components: a consequence relation and a scheme for the connectives. Both notions rely on the definition of a truth value assignment, that is, a function $w: \textit{Var} \longrightarrow \mathcal{V} = \lbrace 0, \sfrac{1}{2}, 1\rbrace$. A truth value assignment is extended to a valuation $v: \mathcal{L} \longrightarrow \mathcal{V}$ through a three-valued \textit{scheme} for the connectives. A scheme is a triple $\langle f^{\neg}, f^{\wedge}, f^{\lor} \rangle$ of operations defined over $\mathcal{V}$. Given a scheme $\mathbf{X}$, a truth value assignment $w$ is extended to an $\mathbf{X}$-valuation $v$ by letting:
\begin{align*}
v(p) &= w(p), \ {\text{if} \ p \in \textit{Var},}\\
v(\neg \phi) &= f^{\neg}v(\phi),\\
v(\phi \wedge \chi) &= f^{\wedge}(v(\phi), v(\chi)),\\
v( \phi \lor \chi) &= f^{\lor}(v(\phi), v(\chi)).\\
\end{align*}
In this paper, we will consider the three-valued schemes satisfying the properties of \textit{Boolean normality} and \textit{monotonicity}.

\begin{Definition}[\citealp{da2023three}]
\label{def:boolean-normality}
An $n$-ary operation $f^\star$ is \textit{Boolean normal} if for $\{a_1,...,a_n\} \subseteq \{0,1\}$, $f^\star(a_1,...,a_n)=f^\star_{\mathsf{CL}}(a_1,...,a_n)$, where $f^\star_{\mathsf{CL}}$ is the corresponding operation over the usual two-element Boolean algebra. 
A scheme is \textit{Boolean normal} if and only if each of its operations is.
\end{Definition}

\noindent
In other words, a scheme is Boolean normal if each of its operations agrees on the classical values with the classical two-valued operations. 

The notion of monotonicity is defined from the information order over $\mathcal{V}$, where it is assumed that the value $\sfrac{1}{2}$ is less informative than $1$ and $0$. If a scheme is monotonic, each of its operations is order-preserving with respect to the informational order.

\begin{Definition}[\citealp{da2023three}]
\label{def:monotonicity}
Given the order $\le_{I}$ defined over $\mathcal{V}$ so that $\sfrac{1}{2} <_{I} 0$ and $\sfrac{1}{2} <_{I} 1$, a scheme is said to be \textit{monotonic} if the following holds for each of its operations $f^\star$ of arity $n$:
$$\textup{If} \ x_i \le_I y_i \ \textup{for all} \ 1 \le i \le n, \ \textup{then} \ f^\star(x_1, \ldots, x_n) \le_{I} f^\star(y_1, \ldots, y_n).$$
\end{Definition}

\noindent
As a consequence, for any formula $\phi$, and two valuations $v$ and $v'$, if $v'$ is more informative than $v$ with respect to the atoms of $\phi$, then $v'$ is more informative than $v$ regarding $\phi$. This is illustrated in the next corollary.

\begin{Corollary}\label{cor:mono}
    Let $\mathbf{X}$ be a monotonic scheme with operations $f^{\star}_1 \ldots f^{\star}_n$ and $v, v'$ be two $\mathbf{X}$-valuations. For all $\phi \in \mathcal{L}$, it holds that if $v(p) \le_I v'(p)$ for all $p \in \mathit{At}(\phi)$, then $v(\phi) \le_I v'(\phi)$.
\end{Corollary}
\begin{proof}
    We proceed by induction on the complexity of $\phi$.

    \medskip
    
    \underline{Base case:} Straightforward.

    \medskip
    
    \underline{Inductive step:} Assume $v(p) \le_I v'(p)$ for all $p \in \mathit{At}(\phi)$ and $\phi = \star (\psi_1, \ldots, \psi_n)$, with $\star$ the connective corresponding to the operation $f^{\star}_i$. By induction hypothesis, $v(\psi_i) \le_I v'(\psi_i)$, since $\mathit{At}(\psi_i) \subseteq \mathit{At}(\phi)$ for all $i \le n$. Given that $\mathbf{X}$ is monotonic, $v(\phi) = v(\star (\psi_1, \ldots, \psi_n)) = f^{\star}_i (v(\psi_1), \ldots, v(\psi_n)) \le_I f^{\star}_i (v^*(\psi_1), \ldots, v^*(\psi_n)) = v^*(\star (\psi_1, \ldots, \psi_n)) = v^*(\phi)$.
\end{proof}

\begin{figure}[b]
\centering
\begin{minipage}{.20\textwidth}
\begin{displaymath}
\begin{array}{c| c}
 & f^{\neg}\\
\hline
1 & 0 \\
\sfrac{1}{2} & \sfrac{1}{2} \\
0 & 1 \\
\end{array}
\end{displaymath}
\end{minipage}
\begin{minipage}{.30\textwidth}
\begin{displaymath}
\begin{array}{c| c c c}
f^{\wedge} & 1 & \sfrac{1}{2} & 0\\
\hline
1 & 1 & \sfrac{1}{2} & 0\\
\sfrac{1}{2} & \sfrac{1}{2} & \sfrac{1}{2} & 0, \sfrac{1}{2}\\
0 & 0 & 0, \sfrac{1}{2} & 0\\
\end{array}
\end{displaymath}
\end{minipage}
\begin{minipage}{.37\textwidth}
 \begin{displaymath}
\begin{array}{c| c c c}
f^{\lor} & 1 & \sfrac{1}{2} & 0\\
\hline
1 & 1 &\sfrac{1}{2}, 1  & 1\\
\sfrac{1}{2} & \sfrac{1}{2}, 1 & \sfrac{1}{2} & \sfrac{1}{2}\\
0 & 1 & \sfrac{1}{2} & 0\\
\end{array}
\end{displaymath}
\end{minipage}
\caption{All the Boolean normal monotonic schemes}
\label{fig:trt}
\end{figure}

By combining the notions of Boolean normality and monotonicity, we obtain the truth tables depicted in Figure \ref{fig:trt}. A comma between two values indicates a possible choice for the truth table of the corresponding operation. To fully define the truth table for a binary connective, one must therefore specify the outputs for the two unsettled cases. For example, in the Strong Kleene scheme, the conjunction is defined so that $f^{\wedge}(0, \sfrac{1}{2}) = 0$, whereas in the Weak Kleene scheme, the conjunction is defined so that $f^{\wedge}(0, \sfrac{1}{2}) = \sfrac{1}{2}$. In total, there are 16 distinct Boolean normal monotonic schemes. Each scheme will be denoted by a pair $\mathbf{X}_\wedge/\mathbf{Y}_\lor$ with $\mathbf{X}_\wedge$ referring to a possible choice of truth values for the two unsettled cases of $f^\wedge$, and $\mathbf{Y}_\lor$ to a possible choice of truth values for the two unsettled cases of $f^\lor$. The name of a particular scheme is obtained by pairing an $\mathbf{X}_\wedge$ from Figure \ref{fig:tvwedge} with a $\mathbf{Y}_\lor$ from Figure \ref{fig:tvlor}. For simplicity, operations are identified with their corresponding symbols in the language.

The Strong Kleene scheme is denoted $\mathbf{SK}_\wedge/\mathbf{SK}_\lor$, this scheme can be alternatively defined by ordering the values of $\mathcal{V}$ as $0 < \sfrac{1}{2} < 1$ and letting $f^{\neg}(x) = 1 - x$, $f^{\wedge}(x, y)= min(x, y)$, and $f^{\lor}(x, y) = max(x, y)$. The Weak Kleene scheme is denoted $\mathbf{WK}_\wedge/\mathbf{WK}_\lor$. In this scheme, the intermediate value is infectious, in the sense that any operation using $\sfrac{1}{2}$ as an input yields $\sfrac{1}{2}$ as an output. The Left Middle Kleene Scheme $\mathbf{LMK}_\wedge/\mathbf{LMK}_\lor$ can be viewed as a middle ground between the two previous schemes. Its truth tables are asymmetric: the intermediate value is infectious on the left of the binary connectives, but not on the right. Dual to this scheme is the Right Middle Kleene scheme, $\mathbf{RMK}_\wedge/\mathbf{RMK}_\lor$, where the binary operations are infectious on the right. Ultimately, all other Boolean normal monotonic schemes can be obtained by combining the truth tables of conjunction and disjunction of the aforementioned schemes.

\begin{figure}
\begin{subfigure}
    \centering
\begin{parcolumns}{4}
\colchunk[1]{$\mathbf{SK}_\wedge$
        \begin{enumerate}[align=left,wide, labelindent=2pt, itemsep=2pt, topsep=2pt]
        \item[] $0 \wedge \sfrac{1}{2} = 0$
        \item[] $\sfrac{1}{2} \wedge 0 = 0$
        \end{enumerate}}

\colchunk[2]{$\mathbf{WK}_\wedge$
    \begin{enumerate}[align=left,wide, labelindent=2pt, itemsep=2pt, topsep=2pt]
        \item[] $0 \wedge \sfrac{1}{2} = \sfrac{1}{2}$
        \item[] $\sfrac{1}{2} \wedge 0 = \sfrac{1}{2}$
    \end{enumerate}}
    
\colchunk[3]{$\mathbf{LMK}_\wedge$
        \begin{enumerate}[align=left,wide, labelindent=2pt, itemsep=2pt, topsep=2pt]
        \item[] $0 \wedge \sfrac{1}{2} = 0$
        \item[] $\sfrac{1}{2} \wedge 0 = \sfrac{1}{2}$
        \end{enumerate}}
    
\colchunk[4]{$\mathbf{RMK}_\wedge$
        \begin{enumerate}[align=left,wide, labelindent=2pt, itemsep=2pt, topsep=2pt]
        \item[] $0 \wedge \sfrac{1}{2} = \sfrac{1}{2}$
        \item[] $\sfrac{1}{2} \wedge 0 = 0$
        \end{enumerate}}

\end{parcolumns}
\caption{Possible choices of truth values for the two unsettled cases of $\wedge$}
\label{fig:tvwedge}
\end{subfigure}

\begin{subfigure}
    \centering
\begin{parcolumns}{4}
\colchunk[1]{$\mathbf{SK}_\lor$
        \begin{enumerate}[align=left,wide, labelindent=2pt, topsep=2pt]
        \item[] $1 \lor \sfrac{1}{2} = 1$
        \item[] $\sfrac{1}{2} \lor 1 = 1$
        \end{enumerate}}
        
\colchunk[2]{$\mathbf{WK}_\lor$
    \begin{enumerate}[align=left,wide, labelindent=2pt, itemsep=2pt, topsep=2pt]
        \item[] $1 \lor \sfrac{1}{2} = \sfrac{1}{2}$
        \item[] $\sfrac{1}{2} \lor 1 = \sfrac{1}{2}$
    \end{enumerate}}

\colchunk[3]{$\mathbf{LMK}_\lor$
    \begin{enumerate}[align=left,wide, labelindent=2pt, itemsep=2pt, topsep=2pt]
        \item[] $1 \lor \sfrac{1}{2} = 1$
        \item[] $\sfrac{1}{2} \lor 1 = \sfrac{1}{2}$
    \end{enumerate}}
    
\colchunk[4]{$\mathbf{RMK}_\lor$
    \begin{enumerate}[align=left,wide, labelindent=2pt, itemsep=2pt, topsep=2pt]
        \item[] $1 \lor \sfrac{1}{2} = \sfrac{1}{2}$
        \item[] $\sfrac{1}{2} \lor 1 = 1$
    \end{enumerate}}
\end{parcolumns}
\caption{Possible choices of truth values for the two unsettled cases of $\lor$}
\label{fig:tvlor}
\end{subfigure}
\end{figure}

Let us now consider the second component of the semantics of a logic. Given a scheme $\mathbf{X}$ and a standard $\mathbf{xy}$ with $\mathbf{x}, \mathbf{y} \subseteq \lbrace 0, \sfrac{1}{2}, 1\rbrace$, the canonical notion of logical consequence associated with this standard is defined as follows: 

\begin{Definition}[$\mathbf{xy}$-satisfaction, $\mathbf{xy}$-validity]
A $\mathbf{X}$-valuation $v$ satisfies an inference $\Gamma \Rightarrow \Delta$ with respect to a standard $\mathbf{xy}$ (symbolized by $v \models_{\mathbf{X}}^{\mathbf{xy}} \Gamma \Rightarrow \Delta$) if and only if $v(\gamma)\in \mathbf{x}$ for all $\gamma \in \Gamma$ only if $v(\delta) \in \mathbf{y}$ for some $\delta \in \Delta$. An inference $ \Gamma \Rightarrow \Delta$ is valid in a logic $\mathsf{L}$ defined by a scheme $\mathbf{X}$ and a $\mathbf{xy}$ consequence relation (symbolized by $\models_{\mathbf{X}}^{\mathbf{xy}} \Gamma \Rightarrow \Delta$) if and only if $v \models_{\mathbf{X}}^{\mathbf{xy}} \Gamma \Rightarrow \Delta$ for all $\mathbf{X}$-valuations $v$.
\end{Definition}
\noindent
Consequence relations will at times be identified with the set of inferences they validate. For example, $\models_{\mathbf{X}}^{\mathbf{xy}} \subseteq \models_{\mathbf{X}'}^{\mathbf{x'y'}}$ is used to indicate that any inference valid under the first consequence relation is also valid under the second. Given the Boolean normality of the schemes discussed here, the classical two-valued logical consequence (denoted by $\models_{\mathsf{CL}}$) can be defined from any scheme by restricting valuations to the set $\lbrace 1, 0 \rbrace$ and using the standards $\mathbf{x} = \mathbf{y} = \lbrace 1 \rbrace$. These restricted valuations will henceforth be referred to as \textit{bivaluations}, in contrast to the unrestricted \textit{trivaluations}.

As noted in the previous section, this paper will focus exclusively on the \textit{intersective mixed} consequence relations \citep{Chemlaetal2017}. There are five such relations, all defined from the sets $\mathbf{s} = \lbrace 1 \rbrace$ (strict) and $\mathbf{t} = \lbrace 1, \sfrac{1}{2} \rbrace$ (tolerant). They naturally correspond to the class of three-valued monotonic consequence relations, meaning that they are the only relations $\models_{\mathbf{X}}^{\mathbf{xy}}$ that satisfy the monotonicity condition: if $\models_{\mathbf{X}}^{\mathbf{xy}} \Gamma \Rightarrow \Delta$, then $\models_{\mathbf{X}}^{\mathbf{xy}} \Gamma, \Gamma' \Rightarrow \Delta, \Delta'$. The first four \textit{intersective mixed} consequence relations are obtained by combining these two sets in every possible way, resulting in the pairs $\mathbf{ss}$, $\mathbf{tt}$, $\mathbf{st}$, and $\mathbf{ts}$. The pairs $\mathbf{ss}$ and $\mathbf{tt}$ define the pure consequence relations, which preserve a fixed set of designated values from premises to conclusions. Standards $\mathbf{st}$ and $\mathbf{ts}$ impose different constraints: $\mathbf{st}$ ensures that it is impossible for the premises to be true and the conclusions to be false, while $\mathbf{ts}$ ensures that the conclusions cannot be non-true if the premises are non-false. The last intersective mixed consequence relation is defined by the standard $\mathbf{ss} \cap \mathbf{tt}$, the intersection of $\mathbf{ss}$ and $\mathbf{tt}$. The notion of validity associated with the latter consequence relation is as follows:

\begin{Definition}[$\mathbf{ss} \cap \mathbf{tt}$-validity]
  Given a scheme $\mathbf{X}$ and the $\mathbf{ss}$ and $\mathbf{tt}$ standards, 
$\models_{\mathbf{X}}^{\mathbf{ss} \cap \mathbf{tt}} \Gamma \Rightarrow \Delta$ if and only if $\models_{\mathbf{X}}^{\mathbf{ss}} \Gamma \Rightarrow \Delta$ and $\models_{\mathbf{X}}^{\mathbf{tt}} \Gamma \Rightarrow \Delta$.
\end{Definition}

The consequence relation defined from the intersection of $\mathbf{ss}$ and $\mathbf{tt}$ is also known in the literature as the \textit{order-theoretic} consequence relation \citep{Field2008, Chemlaetal2017}. The notion of validity associated with this consequence relation is provably equivalent to the following condition: Given the total order $0 < \sfrac{1}{2} < 1$, the value of the conclusions cannot be lower than the value of the premises. 

By associating the intersective mixed consequence relations with the Boolean normal monotonic schemes, one obtains well-known three-valued logics. Combining the standards $\mathbf{ss}$, $\mathbf{tt}$, $\mathbf{st}$, $\mathbf{ts}$ and $\mathbf{ss} \cap \mathbf{tt}$ with the Strong Kleene scheme $\mathbf{SK}_\wedge/\mathbf{SK}_\lor$ produces, respectively, Kleene's logic $\mathsf{K}_3$ \citep{Kleene1952-KLEITM}, Asenjo and Priest's $\mathsf{LP}$ \citep{Asenjo1966,Priest1979}, the strict-tolerant logic $\mathsf{ST}$ \citep{Girard1987, Frankowski2004, Cobrerosetal2012}, the tolerant-strict logic $\mathsf{TS}$ \citep{Malinowski1990} and the logic dubbed \textit{Kalman implication} by \cite{Makinson1973}.

Combining the $\mathbf{ss}$ standard with the Left Middle Kleene scheme yields a logic often referred to as \textit{Middle Kleene} \citep{Peters1979, BeaverKrahmer2001, George2014}, and used to account for presupposition projections. Such a logic has the particularity of having a non-commutative disjunction.\footnote{To see this, note that $v \not\models_{\mathbf{LMK}_{\wedge}/\mathbf{LMK}_{\lor}}^{\mathbf{ss}} p \lor q \Rightarrow q \lor p$ for some $v$ such that $v(p)=1$ and $v(q)= \sfrac{1}{2}$. By a similar argument, it can be shown that combining this scheme with a $\mathbf{tt}$ consequence relation results in a logic with a non-commutative conjunction.}

Combining the standards $\mathbf{ss}$ and $\mathbf{tt}$ with the Weak Kleene scheme respectively produces \cite{Bochvar1938}'s logic $\mathbf{K}_3^w$ and \cite{Hallden1949}'s Paraconsistent Weak Kleene logic $\mathbf{PWK}$. The first has the specificity of invalidating disjunction introduction $\phi \Rightarrow \phi \lor \psi$ and the second conjunction elimination $\phi \wedge \psi \Rightarrow \phi$.\footnote{If $v(p)=1$ and $v(q)=\sfrac{1}{2}$, $v \not\models_{\mathbf{WK}_{\wedge}/\mathbf{WK}_{\lor}}^{\mathbf{ss}} p \Rightarrow p \lor q$. If $v(p)=0$ and $v(q)=\sfrac{1}{2}$, then $v \not\models_{\mathbf{WK}_{\wedge}/\mathbf{WK}_{\lor}}^{\mathbf{tt}} p \wedge q \Rightarrow p$.}

A multitude of other logics, which to the best of our knowledge have not yet been studied, can be defined by combining an intersective mixed consequence relation with a Boolean normal monotonic scheme, particularly schemes of the form $\mathbf{X}_{\wedge}/\mathbf{Y}_{\lor}$ where $\mathbf{X} \neq \mathbf{Y}$. We will defer the discussion of these schemes to a later occasion and now turn our attention to the split interpolation property.

\section{Results independent of a choice of scheme}\label{sec:nri}
The main question we aim to address is the following: for which pair of standards $\mathbf{xy}$ and $\mathbf{x}'\mathbf{y}'$, and Boolean normal monotonic scheme $\mathbf{X}$, does the following property hold?

\vspace{7pt}
If $\mathit{At}(\phi) \cap \mathit{At}(\psi) \neq \emptyset$ and $\models_{\mathsf{CL}} \phi \Rightarrow \psi$, then there is a formula $\chi$ such that each of its atoms is in $\phi$ and $\psi$, and $\models_{\mathbf{X}}^{\mathbf{xy}} \phi \Rightarrow \chi$ and $\models_{\mathbf{X}}^{\mathbf{x}'\mathbf{y}'} \chi \Rightarrow \psi$.
\vspace{7pt}

We will examine first the pairs of logics for which the success or failure of the property can be attributed solely to the choice of pair of standards. We begin by revisiting a significant theorem proved by \cite{da2023three}, which demonstrates that any logic with a $\mathbf{st}$ consequence relation, based on a Boolean normal monotonic scheme, validates the same inferences as classical logic.

\begin{Theorem}[\citealp{da2023three}]\label{thm:dare}
    For every Boolean normal monotonic scheme $\mathbf{X}$, $\models_{\mathbf{X}}^{\mathbf{st}} \ = \ \models_{\mathsf{CL}}$.
\end{Theorem}

\noindent
This result enables us to work within a three-valued framework for the remainder of the paper, avoiding the need to reduce every trivaluation to a bivaluation. It also allows us to prove the first and only positive result independent of a choice of scheme:

\begin{Proposition}\label{prp:st}
    For every standard $\mathbf{xy}$, if $\mathbf{X}$ is a Boolean normal monotonic scheme, the following split interpolation property holds: 
    
\vspace{7pt}
If $\mathit{At}(\phi) \cap \mathit{At}(\psi) \neq \emptyset$ and $\models_{\mathbf{X}}^{\mathbf{st}} \phi \Rightarrow \psi$, then there is a formula $\chi$ such that each of its atoms is in $\phi$ and $\psi$, and $\models_{\mathbf{X}}^{\mathbf{st}} \phi \Rightarrow \chi$ and $\models_{\mathbf{X}}^{\mathbf{st}} \chi \Rightarrow \psi$.
\end{Proposition}

\begin{proof}
It follows directly from Theorem \ref{thm:dare} that this proposition is equivalent to Craig's deductive interpolation theorem.
\end{proof}

The next fact guarantees that for any logic with an $\mathbf{ss}$- or $\mathbf{tt}$-consequence relation based on a scheme $\mathbf{X}$, its set of valid inferences is a subset of the set of valid inferences of the $\mathbf{st}$-consequence based on the same scheme.

\begin{Fact}\label{fct:subst}
    For every Boolean normal monotonic scheme $\mathbf{X}$ and $\mathbf{X}'$, $\models_{\mathbf{X}}^{\mathbf{ss}}, \models_{\mathbf{X}}^{\mathbf{tt}}  \ \subseteq \ \models_{\mathbf{X}'}^{\mathbf{st}}$.
\end{Fact}
\begin{proof}
   Let $\mathbf{X}$ be a Boolean normal monotonic scheme. Assume $\not\models_{\mathbf{X}}^{\mathbf{st}} \Gamma \Rightarrow \Delta$. Then there is an $\mathbf{X}$-valuation $v$ such that $v(\gamma) =1$ for all $\gamma \in \Gamma$ and $v(\delta)=0$ for all $\delta \in \Delta$. Therefore, $\not\models_{\mathbf{X}}^{\mathbf{ss}} \Gamma \Rightarrow \Delta$ and $\not\models_{\mathbf{X}}^{\mathbf{tt}} \Gamma \Rightarrow \Delta$. Now, by Theorem \ref{thm:dare}, $\models_{\mathbf{X}}^{\mathbf{st}} = \models_{\mathbf{X}'}^{\mathbf{st}}$ for all Boolean normal monotonic scheme $\mathbf{X}'$, so $\models_{\mathbf{X}}^{\mathbf{ss}}, \models_{\mathbf{X}}^{\mathbf{tt}}  \ \subseteq \ \models_{\mathbf{X}'}^{\mathbf{st}}$.
\end{proof}

Given Theorem \ref{thm:dare}, this fact implies in particular that classical logic is an extension of any logic with an $\mathbf{ss}$- or $\mathbf{tt}$-consequence relation based on a Boolean normal monotonic scheme. 

Our first failure result can now be stated: For any pair of logics, where both are based on a Boolean normal monotonic scheme, and the first is defined by a $\mathbf{tt}$-consequence relation, the property fails regardless of the choice of standards for the second logic.

\begin{Proposition}\label{prp:tt}
    For every standard $\mathbf{xy}$, if $\mathbf{X}$ is a Boolean normal monotonic scheme, the following split interpolation property does not hold: 
    
\vspace{7pt}
If $\mathit{At}(\phi) \cap \mathit{At}(\psi) \neq \emptyset$ and $\models_{\mathbf{X}}^{\mathbf{st}} \phi \Rightarrow \psi$, then there is a formula $\chi$ such that each of its atoms is in $\phi$ and $\psi$, and $\models_{\mathbf{X}}^{\mathbf{tt}} \phi \Rightarrow \chi$ and $\models_{\mathbf{X}}^{\mathbf{xy}} \chi \Rightarrow \psi$.
\end{Proposition}

\begin{proof}
    We show that $\models_{\mathbf{X}}^{\mathbf{st}} p \lor (q \wedge \neg q) \Rightarrow p$, but there is no $\chi$ such that $\mathit{At}(\chi) \subseteq \mathit{At}(\phi) \cap \mathit{At}(\psi)$, $\models_{\mathbf{X}}^{\mathbf{tt}} p \lor (q \wedge \neg q) \Rightarrow \chi$ and $\models_{\mathbf{X}}^{\mathbf{xy}} \chi \Rightarrow p$. Assume for the sake of contradiction that there is such a $\chi$, and let $v$ be a valuation such that $v(p)= 0$ and $v(q)=\sfrac{1}{2}$. Then $v(p \lor (q \wedge \neg q)) = \sfrac{1}{2}$ since $\mathbf{X}$ is Boolean normal monotonic. Now, $\models_{\mathbf{X}}^{\mathbf{tt}} p \lor (q \wedge \neg q) \Rightarrow \chi$, so $v(\chi) \neq 0$. But $\mathit{At}(\chi)=\lbrace p \rbrace \subseteq \mathit{At}(\phi) \cap \mathit{At}(\psi)$, so $v(\chi) = 1$ since $v(p) \neq \sfrac{1}{2}$ and $\mathbf{X}$ is Boolean normal. Therefore, $v\not\models_{\mathbf{X}}^{\mathbf{xy}} \chi \Rightarrow p$, which is impossible.
\end{proof}

This result can be immediately extended to logics based on the intersection of $\mathbf{ss}$ and $\mathbf{tt}$.

\begin{Corollary}
    Let $\mathbf{X}$ be a Boolean normal monotonic scheme. For every standard $\mathbf{xy}$, $\mathbf{X}$ does not satisfy the following split interpolation property: 

\vspace{7pt}
If $\mathit{At}(\phi) \cap \mathit{At}(\psi) \neq \emptyset$ and $\models_{\mathbf{X}}^{\mathbf{st}} \phi \Rightarrow \psi$, then there is a formula $\chi$ such that each of its atoms is in $\phi$ and $\psi$, and $\models_{\mathbf{X}}^{\mathbf{ss} \cap \mathbf{tt}} \phi \Rightarrow \chi$ and $\models_{\mathbf{X}}^{\mathbf{xy}} \chi \Rightarrow \psi$.
\end{Corollary}
\begin{proof}
    From Proposition \ref{prp:tt} and $\models_{\mathbf{X}}^{\mathbf{ss} \cap \mathbf{tt}} \subseteq \models_{\mathbf{X}}^{\mathbf{tt}}$.
\end{proof}

Dually, for any pair of logics, if the second is defined by an $\mathbf{ss}$-consequence relation, the property fails regardless of the choice of consequence relation for the first logic, and similarly for the intersection of $\mathbf{ss}$ and $\mathbf{tt}$.

\begin{Proposition}\label{prp:ss}
    Let $\mathbf{X}$ be a Boolean normal monotonic scheme. For every standard $\mathbf{xy}$, $\mathbf{X}$ does not satisfy the following split interpolation property: 

\vspace{7pt}
If $\mathit{At}(\phi) \cap \mathit{At}(\psi) \neq \emptyset$ and $\models_{\mathbf{X}}^{\mathbf{st}} \phi \Rightarrow \psi$, then there is a formula $\chi$ such that each of its atoms is in $\phi$ and $\psi$, and $\models_{\mathbf{X}}^{\mathbf{xy}} \phi \Rightarrow \chi$ and $\models_{\mathbf{X}}^{\mathbf{ss}} \chi \Rightarrow \psi$.
\end{Proposition}

\begin{proof}
It is enough to show that $\models_{\mathbf{X}}^{\mathbf{st}} p \Rightarrow p \wedge (q \lor \neg q)$, but there is no $\chi$ such that $\mathit{At}(\chi) \subseteq \mathit{At}(\phi) \cap \mathit{At}(\psi)$, $\models_{\mathbf{X}}^{\mathbf{xy}} p \Rightarrow \chi$ and $\models_{\mathbf{X}}^{\mathbf{ss}} \chi \Rightarrow p \wedge (q \lor \neg q)$. This can be accomplished in a manner similar to the proof of Proposition \ref{prp:tt}. 
\end{proof}

\begin{Corollary}
    Let $\mathbf{X}$ be a Boolean normal monotonic scheme. For every standard $\mathbf{xy}$, $\mathbf{X}$ does not satisfy the following split interpolation property: 

\vspace{7pt}
If $\mathit{At}(\phi) \cap \mathit{At}(\psi) \neq \emptyset$ and $\models_{\mathbf{X}}^{\mathbf{st}} \phi \Rightarrow \psi$, then there is a formula $\chi$ such that each of its atoms is in $\phi$ and $\psi$, and $\models_{\mathbf{X}}^{\mathbf{xy}} \phi \Rightarrow \chi$ and $\models_{\mathbf{X}}^{\mathbf{ss} \cap \mathbf{tt}} \chi \Rightarrow \psi$.
\end{Corollary}
\begin{proof}
    From Proposition \ref{prp:ss} and $\models_{\mathbf{X}}^{\mathbf{ss} \cap \mathbf{tt}} \subseteq \models_{\mathbf{X}}^{\mathbf{ss}}$.
\end{proof}

We proceed with the following propositions showing that the property will also fail when one of the two logics is based on a $\mathbf{ts}$-consequence relation. Since any combination of a Boolean normal monotonic scheme and a $\mathbf{ts}$-consequence relation results in an empty logic (see \citealp{da2023three}), that is, a logic that does not validate any inference, the proofs are immediate.

\begin{Proposition}\label{prp:ts}
    Let $\mathbf{X}$ be a Boolean normal monotonic scheme. For every standard $\mathbf{xy}$, $\mathbf{X}$ does not satisfy the following split interpolation property: 

\vspace{7pt}
If $\mathit{At}(\phi) \cap \mathit{At}(\psi) \neq \emptyset$ and $\models_{\mathbf{X}}^{\mathbf{st}} \phi \Rightarrow \psi$, then there is a formula $\chi$ such that each of its atoms is in $\phi$ and $\psi$, and $\models_{\mathbf{X}}^{\mathbf{ts}} \phi \Rightarrow \chi$ and $\models_{\mathbf{X}}^{\mathbf{xy}} \chi \Rightarrow \psi$.
\end{Proposition}

\begin{proof}
    The proof is straightforward since there are no $\phi$ and $ \chi$ such that $\models_{\mathbf{X}}^{\mathbf{ts}} \phi \Rightarrow \chi$.
\end{proof}

\begin{Proposition}\label{prp:ts2}
    Let $\mathbf{X}$ be a Boolean normal monotonic scheme. For every standard $\mathbf{xy}$, $\mathbf{X}$ does not satisfy the following split interpolation property:
    
\vspace{7pt}
If $\mathit{At}(\phi) \cap \mathit{At}(\psi) \neq \emptyset$ and $\models_{\mathbf{X}}^{\mathbf{st}} \phi \Rightarrow \psi$, then there is a formula $\chi$ such that each of its atoms is in $\phi$ and $\psi$, and $\models_{\mathbf{X}}^{\mathbf{xy}} \phi \Rightarrow \chi$ and $\models_{\mathbf{X}}^{\mathbf{ts}} \chi \Rightarrow \psi$.
\end{Proposition}

\begin{proof}
    The proof is straightforward since there are no $\chi$ and $ \psi$ such that $\models_{\mathbf{X}}^{\mathbf{ts}} \chi \Rightarrow \psi$.
\end{proof}

Figure \ref{tbl:1} summarizes the cases for which the split interpolation fails or holds, independently of a choice of scheme. The cases not addressed---those involving the combination of a $\mathbf{ss}$-consequence relation with a $\mathbf{tt}$- or $\mathbf{st}$-relation, or an $\mathbf{st}$-consequence relation with a $\mathbf{tt}$- one---require a more nuanced approach that involves selecting a specific scheme. We explore these cases in the next section.

\begin{figure}[t]
\begin{center}
\begin{displaymath}
\begin{tabular}{| c | c | c | c |  c | c |}
\hline
 \diagbox{$\mathcal{C}$}{$\mathcal{C'}$} & $\mathbf{ss}$ & $\mathbf{tt}$ & $\mathbf{st}$ & $\mathbf{ts}$ & $\mathbf{ss} \cap \mathbf{tt}$\\
\hline
$\mathbf{ss}$ & \text{\xmark} &  &  & \text{\xmark} & \text{\xmark}\\
 \hline
$\mathbf{tt}$ & \text{\xmark} & \text{\xmark} & \text{\xmark} & \text{\xmark} & \text{\xmark}\\
 \hline
$\mathbf{st}$ & \text{\xmark} &  & \text{\cmark} & \text{\xmark} & \text{\xmark}\\
 \hline
$\mathbf{ts}$ & \text{\xmark} & \text{\xmark} & \text{\xmark} & \text{\xmark} & \text{\xmark}\\
 \hline
$\mathbf{ss} \cap \mathbf{tt}$ & \text{\xmark} & \text{\xmark} & \text{\xmark} & \text{\xmark} & \text{\xmark}\\
\hline
\end{tabular}
\end{displaymath}
\caption{Failure/success of the property independent of a choice of scheme}\label{tbl:1}
\end{center}
\end{figure}

\section{Results dependent on a choice of scheme}\label{sec:rdc}
We will concentrate here on the results that depend on a choice of scheme. Building on the findings from the previous section, we will examine the combination of an $\mathbf{ss}$-consequence relation with a $\mathbf{tt}$- or $\mathbf{st}$-relation, as well as an $\mathbf{st}$-consequence relation with a $\mathbf{tt}$- one. We will focus more particularly on the $\mathbf{ss}/\mathbf{tt}$ combination, as the other combinations can be easily derived from it, as we will demonstrate. The split interpolation property can therefore be  temporarily rephrased as follows:

\vspace{7pt}
If $\mathit{At}(\phi) \cap \mathit{At}(\psi) \neq \emptyset$ and $\models_{\mathbf{X}}^{\mathbf{st}} \phi \Rightarrow \psi$, then there is a formula $\chi$ such that each of its atoms is in $\phi$ and $\psi$, and $\models_{\mathbf{X}}^{\mathbf{ss}} \phi \Rightarrow \chi$ and $\models_{\mathbf{X}}^{\mathbf{tt}} \chi \Rightarrow \psi$.
\vspace{7pt}

The question then becomes: for which Boolean normal monotonic scheme $\mathbf{X}$ does the previous property hold? We will first focus on the positive results and list the schemes for which the property holds, before turning to the negative results in Subsection \ref{ssct:nrds}.

\subsection{Positive results}
We start by defining the key notion of a partial sharpening, adapted from \cite{blomet2024sttsproductsum} and first defined in \cite{scambler2020classical}. Given a Boolean normal valuation $v$, a Boolean normal valuation $v^*$ defined over the same scheme is a partial sharpening of $v$ with respect to a specified set of atoms if it agrees with $v$ on all the
atoms that take a classical value, and possibly disagrees with $v$ on the others.

\begin{Definition}[Partial sharpening]\label{D1}
Let $\mathbf{X}$ be a Boolean normal scheme and $v$ be a $\mathbf{X}$-valuation. A $\mathbf{X}$-valuation $v^*$ is said to be a \textit{partial sharpening} of $v$ with respect to a set of atoms $\Sigma \subseteq \textit{Var}$, if for all $p \in \Sigma$, $v(p) \le_I v^*(p)$.
\end{Definition}

The next two corollaries illustrate crucial properties of the notion of partial sharpening, which will prove useful later.

\begin{Corollary}\label{C1}
If $v^*$ is a partial sharpening of $v$ with respect to $\Sigma$, then $v^*$ is a partial sharpening of $v$ with respect to any $\Theta \subseteq \Sigma$.
\end{Corollary}
\begin{proof} Note that if the condition holds for all $p_1 \in \Sigma$, it also holds for all $p_2 \in \Theta \subseteq \Sigma$.
\end{proof}

\begin{Corollary}\label{C2}
If $v^*$ is a partial sharpening of $v$ with respect to $\Sigma$ and $\Theta$, then $v^*$ is a partial sharpening of $v$ with respect to $\Sigma \cup \Theta$.
\end{Corollary}
\begin{proof} Note that if the condition holds for all $p_1 \in \Sigma$ and for all $p_2 \in \Theta$, it also holds for all $p \in \Sigma \cup \Theta$.
\end{proof}

This brings us to the next lemma, which is intimately linked to the monotonicity of the schemes discussed in this paper. Given a formula $\phi$, if $\phi$ takes a classical value according to a valuation $v$, any valuation $v^*$ exactly like $v$ for all classical values assigned to the atoms of $\phi$ will agree with $v$ on the value of $\phi$.

\begin{Lemma}\label{L1}
Let $\mathbf{X}$ be a Boolean normal monotonic scheme. For all $\mathbf{X}$-valuations $v$, if $v(\phi) \neq \sfrac{1}{2}$, then for all partial sharpenings $v^*$ of $v$ with respect to $\mathit{At}(\phi)$, it holds that $v^*(\phi) = v(\phi)$.
\end{Lemma}

\begin{proof}
Assume $v(\phi) \neq \sfrac{1}{2}$. By the definition of $v^*$, $v(p) \le_I v^*(p)$ for all $p \in \mathit{At}(\phi)$, so by Corollary \ref{cor:mono}, $v(\phi) \le_I v^*(\phi)$ since $\mathbf{X}$ is a monotonic scheme. But $v(\phi) \neq \sfrac{1}{2}$ by assumption, so $v^*(\phi) = v(\phi)$.
\end{proof}

For any Boolean normal monotonic scheme $\mathbf{X}$ and any classically valid inference $\phi \Rightarrow \psi$, the strategy to construct the split interpolant typically runs along the following lines. First, we select all $\mathbf{X}$-valuations that satisfy $\phi$. Next, we construct a conjunction of literals consisting of the atoms of $\phi$ and $\psi$ that take a classical value according to a selected valuation. This valuation must also satisfy the resulting conjunction. The interpolant is then obtained by taking the disjunction of all such conjunctions. There is, nonetheless, a caveat. This strategy is not appropriate for schemes with either a non-commutative disjunction or a non-commutative conjunction. To see this, note for instance that
\begin{align*}
\models_{\mathsf{CL}} (p \wedge q) \lor (r \wedge \neg r) \Rightarrow (r \wedge p) \lor (\neg r \wedge q),
\end{align*}
and
\begin{align*}
\models_{\mathbf{LMK}_\wedge/\mathbf{SK}_\lor}^{\mathbf{ss}} (p \wedge q) \lor (r \wedge \neg r) \Rightarrow (p \wedge q) \lor (p \wedge q \wedge r) \lor (p \wedge q \wedge \neg r).
\end{align*}

\enlargethispage{\baselineskip}
\noindent
The conjunction $(p \wedge q)$ in the consequent corresponds to the valuation $v_1$ such that $v_1(p)=v_1(q) = 1$ and $v_1(r)= \sfrac{1}{2}$. The conjunction $(p \wedge q \wedge r)$ corresponds to the valuation $v_2$ such that $v_2(p) = v_2(q) = v_2(r) = 1$. The conjunction $(p \wedge q \wedge \neg r)$ corresponds to the valuation $v_3$ such that $v_3(p) = v_3(q) = 1$ and $v_3(r) = 0$. These three valuations encompass all the valuations that satisfy the antecedent. Yet, there is $v$ such that
\begin{align*}
v \not\models_{\mathbf{LMK}_\wedge/\mathbf{SK}_\lor}^{\mathbf{tt}} (p \wedge q) \lor (p \wedge q \wedge r) \lor (p \wedge q \wedge \neg r) \Rightarrow (r \wedge p) \lor (\neg r \wedge q),
\end{align*}
namely the valuation $v$ such that $v(p) = \sfrac{1}{2}$ and $v(q)=v(r)=0$. The formula $(p \wedge q) \lor (p \wedge q \wedge r) \lor (p \wedge q \wedge \neg r)$ is therefore not a suitable interpolant.\footnote{Commuting $p$ and $q$ won't be enough of a solution, because 
\begin{align*}
v' \not\models_{\mathbf{LMK}_\wedge/\mathbf{SK}_\lor}^{\mathbf{tt}} (q \wedge p) \lor (q \wedge p \wedge r) \lor (q \wedge p \wedge \neg r) \Rightarrow (r \wedge p) \lor (\neg r \wedge q),
\end{align*}
when $v'$ is such that $v(p)= 0$, $v(q)= \sfrac{1}{2}$ and $v(r)=1$.}

The primary technique we will employ to circumvent this limitation is as follows: if $\mathbf{X}$ is a scheme of the form $\mathbf{X}_\wedge/\mathbf{SK}_\lor$, any conjunction in the interpolant will be translated into its De Morgan equivalent. Similarly, for any scheme $\mathbf{X}$ of the form $\mathbf{SK}_\wedge/\mathbf{Y}_\lor$, any disjunction in the interpolant will be translated into its De Morgan equivalent.

\begin{Definition}\label{DefConj} Let $\phi, \psi \in \mathcal{L}$, $\mathbf{X}$ be a Boolean normal monotonic scheme, and $v$ an $\mathbf{X}$-valuation. If $v$ is such that $v(\phi)=1$ and $\mathit{At}(\phi) \cap \mathit{At}(\psi) \neq \emptyset$, let 

\begin{itemize}
 \setlength\itemsep{1em}
\item $C_{\phi, \psi}^{v} = \bigwedge_{v(\alpha) \neq \sfrac{1}{2}} \alpha^{\sim}$
\item $D_{\phi, \psi}^{v} = \neg \bigvee_{v(\alpha) \neq \sfrac{1}{2}} \neg \alpha^{\sim}$ 
\end{itemize}

\noindent
with $\alpha \in \mathit{At}(\phi) \cap \mathit{At}(\psi)$ and $\alpha^\sim=\alpha$ if $v(\alpha)=1$ and $\alpha^\sim=\neg \alpha$ if $v(\alpha)=0$.
\end{Definition}

\begin{Definition}\label{DefDNF}
    Given two formulas $\phi$ and $\psi$ such that $\mathit{At}(\phi) \cap \mathit{At}(\psi) \neq \emptyset$, let
    
    \begin{itemize}
     \setlength\itemsep{1em}
    \item $\mathtt{C}_{\phi, \psi} = \neg  \bigwedge \lbrace \neg C_{\phi, \psi}^v : v(\phi)=1\rbrace$
    \item $\mathtt{D}_{\phi, \psi} = \bigvee \lbrace D_{\phi, \psi}^v : v(\phi)=1 \rbrace$
    \end{itemize}

    \noindent
    when there is $v^*$ such that $v^*(\phi)=1$, and let $\mathtt{C}_{\phi, \psi} = \mathtt{D}_{\phi, \psi} = p \wedge \neg p$ for some $p \in \mathit{At}(\phi) \cap \mathit{At}(\psi)$ otherwise.
\end{Definition}

The next lemma will allow us to establish the first part of the main theorem of the subsection. Provided that $\phi$ and $\psi$ have some atom in common, it shows that for any Boolean normal monotonic scheme $\mathbf{X}$ of the form $\mathbf{SK}_\wedge/\mathbf{Y}_\lor$, $\mathbf{X}_\wedge/\mathbf{SK}_\lor$ or $\mathbf{WK}_\wedge/\mathbf{WK}_\lor$, the interpolant generated by $\phi$ and $\psi$ is logically equivalent to $\phi$.

\begin{Lemma}\label{lem:cd}
Let $\phi$ and $ \psi$ be two formulas such that $\mathit{At}(\phi) \cap \mathit{At}(\psi) \neq \emptyset$. Then
\begin{align*}
\models_{\mathbf{SK}_\wedge/\mathbf{Y}_\lor}^{\mathbf{ss}} \phi \Rightarrow \mathtt{C}_{\phi, \psi} \ &\text{and} \ \models_{\mathbf{SK}_\wedge/\mathbf{Y}_\lor}^{\mathbf{ss}} \mathtt{C}_{\phi, \psi} \Rightarrow \phi,\\
\models_{\mathbf{X}_\wedge/\mathbf{SK}_\lor}^{\mathbf{ss}} \phi \Rightarrow \mathtt{D}_{\phi, \psi} \ &\text{and} \ \models_{\mathbf{X}_\wedge/\mathbf{SK}_\lor}^{\mathbf{ss}} \mathtt{D}_{\phi, \psi} \Rightarrow \phi,\\
\models_{\mathbf{WK}_\wedge/\mathbf{WK}_\lor}^{\mathbf{ss}} \phi \Rightarrow \mathtt{D}_{\phi, \psi} \ &\text{and} \ \models_{\mathbf{WK}_\wedge/\mathbf{WK}_\lor}^{\mathbf{ss}} \mathtt{D}_{\phi, \psi} \Rightarrow \phi.\\
\end{align*}
\end{Lemma}
\begin{proof}
For $\models_{\mathbf{SK}_\wedge/\mathbf{Y}_\lor}^{\mathbf{ss}} \phi \Rightarrow \mathtt{C}_{\phi, \psi}$, assume $v(\phi)=1$. Then $C_{\phi, \psi}^v$ is defined since $\mathit{At}(\phi) \cap \mathit{At}(\psi) \neq \emptyset$. By construction and by the $\mathbf{X}$ evaluation of a conjunction, $v(C_{\phi, \psi}^v)=1$. Now, by the $\mathbf{SK}_{\wedge}$ truth table, $v(\mathtt{C}_{\phi, \psi}^v)=1$ as well.  The other left-side validities can be treated similarly.

Turning to $\models_{\mathbf{SK}_\wedge/\mathbf{Y}_\lor}^{\mathbf{ss}} \mathtt{C}_{\phi, \psi} \Rightarrow \phi$, if there is a $v'$ such that $v'(\mathtt{C}_{\phi, \psi})=1$, then for some $v$ such that $v(\phi)=1$, we have $v'(C_{\phi, \psi}^v)=1$ given the $\mathbf{SK}$ truth table for conjunction. Then $v'$ must agree on the atoms of $\phi$ to which $v$ gives a classical value, and possibly differ on the atoms of $\phi$ to which $v$ gives the value $\sfrac{1}{2}$. Thus, $v'$ is a partial sharpening of $v$ with respect to $\mathit{At}(\phi)$, and therefore, since $v(\phi)=1$, we also have $v'(\phi)=1$ by Lemma \ref{L1}. By analogous reasoning, we can prove that $\models_{\mathbf{X}_\wedge/\mathbf{SK}_\lor}^{\mathbf{ss}} \mathtt{D}_{\phi, \psi} \Rightarrow \phi$ and $\models_{\mathbf{WK}_\wedge/\mathbf{WK}_\lor}^{\mathbf{ss}} \mathtt{D}_{\phi, \psi} \Rightarrow \phi$.
\end{proof}

Although the split interpolation property requires that for any interpolated inference $\phi \Rightarrow \psi$, we have $\mathit{At}(\phi) \cap \mathit{At}(\psi) \neq \emptyset$, it is still possible for $\phi$ to be a contradiction or $\psi$ a tautology. For example, the inference $p \wedge \neg p \Rightarrow p \lor \neg p$ is classically valid, and $\mathit{At}(p \wedge \neg p) \cap \mathit{At}(p \lor \neg p) \neq \emptyset$. It is straightforward to verify that the case where $\psi$ is a tautology is already handled by Lemma \ref{lem:cd}, since either $\mathtt{C}_{\phi, \psi}$ or $\mathtt{D}_{\phi, \psi}$ is entailed by $\phi$ and, in turn, entails the tautology $\psi$ for any $\mathbf{tt}$-consequence relation based on a Boolean normal monotonic scheme. However, when $\phi$ is a contradiction, there is no valuation $v$ such that $v(\phi) = 1$, making both $\mathtt{C}_{\phi, \psi}$ and $\mathtt{D}_{\phi, \psi}$ unsuitable as interpolants. In what follows, we show how to construct an appropriate interpolant for this scenario.

\begin{Definition}\label{DefDij}
Let $\phi, \psi \in \mathcal{L}$, $\mathbf{X}$ be a Boolean normal monotonic scheme, and $v$ an $\mathbf{X}$-valuation. If $v$ is such that $v(\psi)=0$ and $\mathit{At}(\phi) \cap \mathit{At}(\psi) \neq 0$, let 

\begin{itemize}
\setlength\itemsep{1em}
\item $E_{\phi, \psi}^{v} = \bigvee_{v(\alpha) \neq \sfrac{1}{2}} \alpha^{\sim}$
\item $F_{\phi, \psi}^{v} = \neg \bigwedge_{v(\alpha) \neq \sfrac{1}{2}} \neg \alpha^{\sim}$ 
\end{itemize}

\noindent
with $\alpha \in \mathit{At}(\phi) \cap \mathit{At}(\psi)$ and $\alpha^\sim=\alpha$ if $v(\alpha)=0$ and $\alpha^\sim=\neg \alpha$ if $v(\alpha)=1$.
\end{Definition}

\begin{Definition}\label{DefCNF}
    Given two formulas $\phi$ and $\psi$ such that $\mathit{At}(\phi) \cap \mathit{At}(\psi) \neq \emptyset$, let
    
\begin{itemize}
 \setlength\itemsep{1em}
    \item $\mathtt{E}_{\phi, \psi} = \neg \bigvee \lbrace \neg E_{\phi, \psi}^v :v(\psi)=0 \rbrace$
    \item $\mathtt{F}_{\phi, \psi} = \bigwedge \lbrace F_{\phi, \psi}^v : v(\psi)=0 \rbrace$
    \end{itemize}

    \noindent
    when there is $v^*$ such that $v^*(\psi)=0$, and let $\mathtt{F}_{\phi, \psi} = \mathtt{E}_{\phi, \psi} = p \lor \neg p$ for some $p \in \mathit{At}(\phi) \cap \mathit{At}(\psi)$ otherwise.
\end{Definition}

\begin{Lemma}\label{lem:ef}
Let $\phi$ and $ \psi$ be two formulas such that $\mathit{At}(\phi) \cap \mathit{At}(\psi) \neq \emptyset$. Then
\begin{align*}
\models_{\mathbf{X}_\wedge/\mathbf{SK}_\lor}^{\mathbf{tt}} \mathtt{E}_{\phi, \psi} \Rightarrow \psi,\\
\models_{\mathbf{SK}_\wedge/\mathbf{Y}_\lor}^{\mathbf{tt}} \mathtt{F}_{\phi, \psi} \Rightarrow \psi,\\
\models_{\mathbf{WK}_\wedge/\mathbf{WK}_\lor}^{\mathbf{tt}} \mathtt{E}_{\phi, \psi} \Rightarrow \psi.
\end{align*}
\end{Lemma}
\begin{proof}
    For $\models_{\mathbf{X}_\wedge/\mathbf{SK}_\lor}^{\mathbf{tt}} \mathtt{E}_{\phi, \psi} \Rightarrow \psi$, assume $v(\psi)=0$. Then $E_{\phi, \psi}^v$ is defined since $\mathit{At}(\phi) \cap \mathit{At}(\psi) \neq \emptyset$. By construction and by the $\mathbf{X}$ evaluation of a disjunction, $v(E_{\phi, \psi}^v)=0$. Now, given the $\mathbf{SK_{\lor}}$ truth tables, $v(\mathtt{E}_{\phi, \psi}^v)=0$ as well. The other validities can be treated similarly.
\end{proof}

With these definitions and lemmas at hand, we can now state the main result of the subsection. 

\begin{Theorem}\label{thm:psds}
If a scheme $\mathbf{X}$ is a scheme of the form $\mathbf{SK}_{\wedge}/\mathbf{Y}_{\lor}$ or $\mathbf{X}_{\wedge}/\mathbf{SK}_{\lor}$ or $\mathbf{WK}_{\wedge}/\mathbf{WK}_{\lor}$ for some $\mathbf{X}_{\wedge}$ and $\mathbf{Y}_{\lor}$, then it satisfies the following split interpolation property: 

\vspace{7pt}
If $\mathit{At}(\phi) \cap \mathit{At}(\psi) \neq \emptyset$ and $\models_{\mathbf{X}}^{\mathbf{st}} \phi \Rightarrow \psi$, then there is a formula $\chi$ such that each of its atoms is in $\phi$ and $\psi$, and $\models_{\mathbf{X}}^{\mathbf{ss}} \phi \Rightarrow \chi$ and $\models_{\mathbf{X}}^{\mathbf{tt}} \chi \Rightarrow \psi$.
\end{Theorem}

\begin{proof}
Assume that $\models_{\mathbf{X}}^{\mathbf{st}} \phi \Rightarrow \psi$. We first show that the property holds for any scheme $\mathbf{X}=\mathbf{SK}_{\wedge}/\mathbf{Y}_{\lor}$. We start with the limit case. If there is no $v$ such that $v(\phi)=1$. We take as interpolant $\chi:=\mathtt{F}_{\phi, \psi}$. Obviously $\models_{\mathbf{X}}^{\mathbf{ss}} \phi \Rightarrow \mathtt{F}_{\phi, \psi}$ and $\models_{\mathbf{X}}^{\mathbf{tt}} \mathtt{F}_{\phi, \psi} \Rightarrow \psi$ by Lemma \ref{lem:ef}. 

Addressing the main case, in which there is a $v$ such that $v(\phi)=1$ let $\chi:=\mathtt{C}_{\phi, \psi}$.  By Lemma \ref{lem:cd}, it directly holds that $\models_{\mathbf{X}}^{\mathbf{ss}} \phi \Rightarrow \mathtt{C}_{\phi, \psi}$. Let us now turn to $\models_{\mathbf{X}}^{\mathbf{tt}} \mathtt{C}_{\phi, \psi} \Rightarrow \psi$.
    Assume $\not\models_{\mathbf{X}}^{\mathbf{tt}} \mathtt{C}_{\phi, \psi} \Rightarrow \psi$ for the sake of contradiction. It follows that there is $v$ such that $v(\mathtt{C}_{\phi, \psi})\neq 0$ and $v(\psi) = 0$. If $v(C_{\phi, \psi}^v) = \sfrac{1}{2}$, there is $v'$ such that $v(C_{\phi, \psi}^{v'}) = \sfrac{1}{2}$. So, for all conjuncts $\alpha^\sim$ of $C_{\phi, \psi}^{v'}$, $v(\alpha^\sim) \neq 0$. Consider $v^*$ which is exactly like $v$ except that for all the conjuncts $\alpha^\sim$ of $C_{\phi, \psi}^{v'}$, if $v(\alpha^\sim) = \sfrac{1}{2}$, then $v^*(\alpha^\sim) = 1$ and if $\beta \in \mathit{At}(\phi) - \mathit{At}(\psi)$, $v^*(\beta) = v'(\beta)$. Since $\mathit{At}(C_{\phi, \psi}^{v'})\cap (\mathit{At}(\phi) - \mathit{At}(\psi)) = \emptyset$, $v^*$ is well-defined. Now, $v^*(C_{\phi, \psi}^{v'}) = 1$, and given the $\mathbf{SK}$ truth table for conjunction, $v^*(\mathtt{C}_{\phi, \psi}) = 1$. Moreover, $v^*$ is a partial sharpening of $v$ with respect to $\mathit{At}(C_{\phi, \psi}^{v'})$ by construction, and is exactly like $v$ for all $\beta \in \mathit{At}(\psi) - \mathit{At}(C_{\phi, \psi}^{v'})$, so trivially it is also a partial sharpening of $v$ with respect to $\mathit{At}(\psi) - \mathit{At}(C_{\phi, \psi}^{v'})$. Therefore, by Corollary \ref{C2}, $v^*$ is a partial sharpening of $v$ with respect to $\mathit{At}(C_{\phi, \psi}^{v'}) \cup (\mathit{At}(\psi) - \mathit{At}(C_{\phi, \psi}^{v'}))$, and by Corollary \ref{C1} to $\mathit{At}(\psi) \subseteq \mathit{At}(C_{\phi, \psi}^{v'}) \cup (\mathit{At}(\psi) - \mathit{At}(C_{\phi, \psi}^{v'}))$. Given that $v(\psi) = 0$, by Lemma \ref{L1}, it follows that $v^*(\psi) = 0$. Let us now consider the value of $\phi$ under $v^*$. $v^*$ is exactly like $v'$ for all $\alpha \in \mathit{At}(\phi) - \mathit{At}(\psi)$ and for all $\beta \in \mathit{At}(C_{\phi, \psi}^{v'})$ as well, since $C_{\phi, \psi}^{v'}$ is a conjunction of literals and $v'(C_{\phi, \psi}^{v'})=1$. If $\alpha \in (\mathit{At}(\phi) \cap \mathit{At}(\psi)) -  \mathit{At}(C_{\phi, \psi}^{v'})$, then $v'(\alpha) = \sfrac{1}{2}$, so $v^*$ is a partial sharpening of $v'$ with respect to $\mathit{At}(\phi)$. Hence, since $v'(\phi)=1$, $v^*(\phi)=1$ by Lemma \ref{L1}. But $v^*(\psi) = 0$, so $\not\models_{\mathbf{X}}^{\mathbf{st}} \phi \Rightarrow \psi$, which is impossible. If $v(\mathtt{C}_{\phi, \psi}) = 1$, define $v^*$ exactly like $v$, except that if $\beta \in \mathit{At}(\phi) - \mathit{At}(\psi)$, $v^*(\beta) = v'(\beta)$. The remaining of the proof is similar to the previous case.

    The case of $\mathbf{X}=\mathbf{X}_{\wedge}/\mathbf{SK}_{\lor}$ can be treated similarly by taking $\mathtt{E}_{\phi, \psi}$ as an interpolant for the limit case and $\mathtt{D}_{\phi, \psi}$ for the main case. As for the case of $\mathbf{X}=\mathbf{WK}_{\wedge}/\mathbf{WK}_{\lor}$, both $\mathtt{E}_{\phi, \psi}$ and $\mathtt{F}_{\phi, \psi}$ are suitable interpolants for the limit case and $\mathtt{C}_{\phi, \psi}$ and $\mathtt{D}_{\phi, \psi}$ for the main case. To see this, note first that given the $\mathbf{WK}_{\wedge}/\mathbf{WK}_{\lor}$ truth tables, for all $\phi$, if $v(\phi) \neq \sfrac{1}{2}$, every $\alpha \in \mathit{At}(\phi)$ is such that $v(\alpha) \neq \sfrac{1}{2}$. Now, to prove that $\models_{\mathbf{X}}^{\mathbf{ss}} \phi \Rightarrow \mathtt{C}_{\phi, \psi}$, one has to assume that $v(\phi)=1$, which entails that $v(\alpha) \neq \sfrac{1}{2}$ for every $\alpha \in \mathit{At}(\phi)$. This case therefore reduces to the case proved above, for $\textup{
At}(\mathtt{C}_{\phi, \psi}) \subseteq \mathit{At}(\phi)$ and all the schemes discussed here are Boolean normal. Similarly, to prove that $\models_{\mathbf{X}}^{\mathbf{tt}} \mathtt{C}_{\phi, \psi} \Rightarrow \psi$, we assume that $v(\mathtt{C}_{\phi, \psi})\neq 0$ and $v(\psi)=0$, and by Boolean normality, this part of the proof reduces to the previous case.
\end{proof}

Given the definition of the consequence relation of $\mathsf{K}_3$ and $\mathsf{LP}$ as the combination of the Strong Kleene $\mathbf{SK}_{\wedge}/\mathbf{SK}_{\lor}$ and the $\mathbf{ss}$ and $\mathbf{tt}$ standards, respectively, the split interpolation property holds for this pair of logics. A result already proven in \cite{milne2016refinement}. However, this result is now extended
to any pair of logics based on a $\mathbf{ss}$- and $\mathbf{tt}$-consequence relation, respectively, with a scheme incorporating either a $\mathbf{SK}$ conjunction or a $\mathbf{SK}$ disjunction, or both a $\mathbf{WK}$ conjunction and a $\mathbf{WK}$ disjunction. Hence, the property also characterizes, for instance, the pair of logics $\mathsf{K}_3^w$ and $\mathsf{PWK}$, whose consequence relations are obtained by combining the Weak Kleene scheme $\mathbf{WK}{\wedge}/\mathbf{WK}_{\lor}$ with the $\mathbf{ss}$ and $\mathbf{tt}$ standards.

It is now straightforward to see from Fact \ref{fct:subst} that all the aforementioned results still hold when one replaces the $\mathbf{ss}$ or $\mathbf{tt}$ standards by an $\mathbf{st}$ standard. This fact can be formally stated as follows:

\begin{Corollary}\label{cor:posst+}
    If a scheme $\mathbf{X}$ is a scheme of the form $\mathbf{SK}_{\wedge}/\mathbf{Y}_{\lor}$ or $\mathbf{X}_{\wedge}/\mathbf{SK}_{\lor}$ or $\mathbf{WK}_{\wedge}/\mathbf{WK}_{\lor}$ for some $\mathbf{X}_{\wedge}$ and $\mathbf{Y}_{\lor}$, then it satisfies the following split interpolation property: 

\vspace{7pt}
If $\mathit{At}(\phi) \cap \mathit{At}(\psi) \neq \emptyset$ and $\models_{\mathbf{X}}^{\mathbf{st}} \phi \Rightarrow \psi$, then

\begin{itemize}
\setlength\itemsep{1em}
\item there is a formula $\chi$ such that each of its atoms is in $\phi$ and $\psi$, and $\models_{\mathbf{X}}^{\mathbf{ss}} \phi \Rightarrow \chi$ and $\models_{\mathbf{X}}^{\mathbf{st}} \chi \Rightarrow \psi$,
\item there is a formula $\chi$ such that each of its atoms is in $\phi$ and $\psi$, and $\models_{\mathbf{X}}^{\mathbf{st}} \phi \Rightarrow \chi$ and $\models_{\mathbf{X}}^{\mathbf{tt}} \chi \Rightarrow \psi$,
\item there is a formula $\chi$ such that each of its atoms is in $\phi$ and $\psi$, and $\models_{\mathbf{X}}^{\mathbf{st}} \phi \Rightarrow \chi$ and $\models_{\mathbf{X}}^{\mathbf{st}} \chi \Rightarrow \psi$.
\end{itemize}
\end{Corollary}
\begin{proof}
    This follows directly from Fact \ref{fct:subst} and Theorem \ref{thm:psds}.
\end{proof}

It is worth noting that, given the equivalence of classical logic with any logic featuring a $\mathbf{st}$-consequence relation based on a Boolean normal monotonic scheme, as stated in Theorem \ref{thm:dare}, the final instance of the split interpolation theorem is simply an alternative formulation of Craig's deductive interpolation theorem, already stated in the previous section. The first two formulations are more refined than the third, as they entail it. However, the formulation of the split interpolation property of Theorem \ref{thm:psds} is ultimately the most refined, as it entails all. We will revisit this in Section \ref{sec:dis}.

\subsection{Negative results}\label{ssct:nrds}
We now turn to the negative results. The subsequent lemma will help us to prove the main result of the subsection. Namely that the property fails for the pair of logics with an $\mathbf{ss}$ and $\mathbf{tt}$ consequence relations based on a scheme in
\begin{align*}
\mathcal{S} = \ &\lbrace{\mathbf{LMK}_{\wedge}/\mathbf{LMK}_{\lor}},\ {\mathbf{LMK}_{\wedge}/\mathbf{RMK}_{\lor}},\ {\mathbf{LMK}_{\wedge}/\mathbf{WK}_{\lor}},\ {\mathbf{RMK}_{\wedge}/\mathbf{LMK}_{\lor}},\\ 
&{\mathbf{RMK}_{\wedge}/\mathbf{RMK}_{\lor}},\ {\mathbf{RMK}_{\wedge}/\mathbf{WK}_{\lor}},\ {\mathbf{WK}_{\wedge}/\mathbf{LMK}_{\lor}},\ {\mathbf{WK}_{\wedge}/\mathbf{RMK}_{\lor}}\rbrace.
\end{align*}

\begin{Lemma}\label{lem:ner}
    Let $v$ and $v'$ be two valuations based on the same scheme in $\mathcal{S}$. Let $v(p)=\sfrac{1}{2}$, and $v'(q)=\sfrac{1}{2}$. For all $\phi \in \mathcal{L}$, if $\phi$ is composed only of the atoms $p$ and $q$, then $v(\phi) = \sfrac{1}{2}$ or $v'(\phi) = \sfrac{1}{2}$.
\end{Lemma}
\begin{proof}
The proof is by induction on the length of $\phi$.

    \medskip

    \underline{Base case:} Straightforward.

    \medskip

    \underline{Inductive step:} 
    \begin{itemize}[align=left]
    \setlength\itemsep{1em}
        \item[($\phi = \neg \psi$).] If $\neg \psi$ is composed only of $p$ and $q$, then $\psi$ as well. So $v(\psi) = \sfrac{1}{2}$ or $v'(\psi) = \sfrac{1}{2}$ by IH. In both cases, $v(\neg \psi) = \sfrac{1}{2}$ or $v'(\neg \psi) = \sfrac{1}{2}$.

        \item[($\phi = \psi \wedge \chi$).] If $\psi \wedge \chi$ is composed only of $p$ and $q$, then $\psi$ and $\chi$ are each at least composed of $p$ or $q$. Hence, $v(\psi) = \sfrac{1}{2}$ or $v'(\psi) = \sfrac{1}{2}$ on the one hand, and $v(\chi) = \sfrac{1}{2}$ or $v'(\chi) = \sfrac{1}{2}$ on the other. We start by assuming that $v(\psi) = \sfrac{1}{2}$, the case for $v'(\psi) = \sfrac{1}{2}$ being dual. If $v(\psi) = \sfrac{1}{2}$, $v(\psi \wedge \chi) = \sfrac{1}{2}$ in any scheme of the form $\mathbf{WK}_{\wedge}/\mathbf{Y}_{\lor}$ or $\mathbf{LMK}_{\wedge}/\mathbf{Y}_{\lor}$. If in addition $v(\chi) = \sfrac{1}{2}$, $v(\psi \wedge \chi) = \sfrac{1}{2}$ in all schemes. If on the other hand $v'(\chi)=\sfrac{1}{2}$, then $v'(\psi \wedge \chi) = \sfrac{1}{2}$ in $\mathbf{RMK}_{\wedge}/\mathbf{Y}_{\lor}$ for any scheme $\mathbf{Y}$. 

        \item[($\phi = \psi \lor \chi$).] If $\psi \lor \chi$ is composed only of $p$ and $q$, then $\psi$ and $\chi$ are each at least composed of $p$ or $q$. Hence, $v(\psi) = \sfrac{1}{2}$ or $v'(\psi) = \sfrac{1}{2}$ on the one hand, and $v(\chi) = \sfrac{1}{2}$ or $v'(\chi) = \sfrac{1}{2}$ on the other. We start by assuming that $v(\psi) = \sfrac{1}{2}$, the case for $v'(\psi) = \sfrac{1}{2}$ being dual. If $v(\psi) = \sfrac{1}{2}$, $v(\psi \lor \chi) = \sfrac{1}{2}$ in any scheme of the form $\mathbf{X}_{\wedge}/\mathbf{WK}_{\lor}$ or $\mathbf{X}_{\wedge}/\mathbf{LMK}_{\lor}$. If in addition $v(\chi) = \sfrac{1}{2}$, $v(\psi \lor \chi) = \sfrac{1}{2}$ in all schemes. If on the other hand $v'(\chi)=\sfrac{1}{2}$, then $v'(\psi \lor \chi) = \sfrac{1}{2}$ in $\mathbf{X}_{\wedge}/\mathbf{RMK}_{\lor}$ for any scheme $\mathbf{X}$.
    \end{itemize}
\end{proof}

For each scheme in $\mathcal{S}$, we begin by showing that, for some classically valid inference, no interpolant exists for the pair of logics with $\mathbf{st}$ and $\mathbf{tt}$ consequence relations based on this scheme. The inference serving as a counterexample involves two atoms. By applying Lemma \ref{lem:ner}, we show that constructing an interpolant for this pair of logics is impossible.

\begin{Theorem}\label{thm:nett}
If a scheme $\mathbf{X}$ is in $\mathcal{S}$, then it does not satisfy the following split interpolation property: 

\vspace{7pt}
If $\mathit{At}(\phi) \cap \mathit{At}(\psi) \neq \emptyset$ and $\models_{\mathbf{X}}^{\mathbf{st}} \phi \Rightarrow \psi$, then there is a formula $\chi$ such that each of its atoms is in $\phi$ and $\psi$, and $\models_{\mathbf{X}}^{\mathbf{st}} \phi \Rightarrow \chi$ and $\models_{\mathbf{X}}^{\mathbf{tt}} \chi \Rightarrow \psi$.
\end{Theorem}
\begin{proof}
We show that for every scheme $\mathbf{X}$ in $\mathcal{S}$, there are $\phi, \psi$ such that $\models_{\mathbf{X}}^{\mathbf{st}} \phi \Rightarrow \psi$, but there is no $\chi$ such that $\mathit{At}(\chi) \subseteq \mathit{At}(\phi) \cap \mathit{At}(\psi)$, $\models_{\mathbf{X}}^{\mathbf{st}} \phi \Rightarrow \chi$ and $\models_{\mathbf{X}}^{\mathbf{tt}} \chi \Rightarrow \psi$. For every scheme $\mathbf{X} \in \mathcal{S}$, we define $\phi = p \wedge q$ and

\begin{itemize}
\setlength\itemsep{1em}
    \item $\psi = (q \lor r) \wedge (p \lor \neg r)$, if $\mathbf{X} = \mathbf{X}_{\wedge}/\mathbf{LMK}_{\lor}$ for some $\mathbf{X}_{\wedge}$,
    \item $\psi = (r \lor q) \wedge (\neg r \lor p)$, if $\mathbf{X} = \mathbf{X}_{\wedge}/\mathbf{RMK}_{\lor}$ for some $\mathbf{X}_{\wedge}$,
    \item $\psi = (r \wedge q) \lor (\neg r \wedge p)$, if $\mathbf{X} = \mathbf{LMK}_{\wedge}/\mathbf{Y}_{\lor}$ for some $\mathbf{Y}_{\lor}$,
    \item $\psi = (q \wedge r) \lor (p \wedge \neg r)$, if $\mathbf{X} = \mathbf{RMK}_{\wedge}/\mathbf{Y}_{\lor}$ for some $\mathbf{Y}_{\lor}$.
\end{itemize}

It can easily be checked that, for each of these cases, $\models_{\mathbf{X}}^{\mathbf{st}} \phi \Rightarrow \psi$. We now show that there is no $\chi$ such that $\mathit{At}(\chi) \subseteq \mathit{At}(\phi) \cap \mathit{At}(\psi)$ and $\models_{\mathbf{X}}^{\mathbf{tt}} \chi \Rightarrow \psi$.

We only prove the first case, the others being similar. Let $v$ be a valuation such that $v(p)=0$, $v(q)= \sfrac{1}{2}$, $v(r)=0$ and let $v'$ be a valuation such that $v'(p)=\sfrac{1}{2}$, $v'(q)= 0$, $v'(r)=1$. Then $v((q \lor r) \wedge (p \lor \neg r))= 0$ and by Lemma \ref{lem:ner}, for all $\chi$ such that $\mathit{At}(\chi) \subseteq \lbrace p, q \rbrace = \mathit{At}(\phi) \cap \mathit{At}(\psi)$, $v(\chi) = \sfrac{1}{2}$ or $v'(\chi) = \sfrac{1}{2}$. Hence, for all $\chi$, $v \not\models_{\mathbf{X}}^{\mathbf{tt}} \chi \Rightarrow \psi$ or $v' \not\models_{\mathbf{X}}^{\mathbf{tt}} \chi \Rightarrow \psi$, and therefore there is no $\chi$ such that $\mathit{At}(\chi) \subseteq \mathit{At}(\phi) \cap \mathit{At}(\psi)$ and $\models_{\mathbf{X}}^{\mathbf{tt}} \chi \Rightarrow \psi$, which, in turn, entails that there is no $\chi$ such that $\mathit{At}(\chi) \subseteq \mathit{At}(\phi) \cap \mathit{At}(\psi)$, $\models_{\mathbf{X}}^{\mathbf{st}} \phi \Rightarrow \chi$ and $\models_{\mathbf{X}}^{\mathbf{tt}} \chi \Rightarrow \psi$, as intended.
\end{proof}

\begin{Corollary}
    If a scheme $\mathbf{X}$ is in $\mathcal{S}$, then it does not satisfy the following split interpolation property: 

\vspace{7pt}
If $\mathit{At}(\phi) \cap \mathit{At}(\psi) \neq \emptyset$ and $\models_{\mathbf{X}}^{\mathbf{st}} \phi \Rightarrow \psi$, then there is a formula $\chi$ such that each of its atoms is in $\phi$ and $\psi$, and $\models_{\mathbf{X}}^{\mathbf{st}} \phi \Rightarrow \chi$ and $\models_{\mathbf{X}}^{\mathbf{tt}} \chi \Rightarrow \psi$.
\end{Corollary}

\begin{proof}
    By Fact \ref{fct:subst}, for any $\mathbf{X}$, $\models_{\mathbf{X}}^{\mathbf{ss}} \ \subseteq \ \models_{\mathbf{X}}^{\mathbf{st}}$, so any counterexample to the $\mathbf{st}$-validity of an inference is also a counterexample to the $\mathbf{ss}$-validity of an inference.
\end{proof}

\begin{figure}[b]
\begin{center}
 \begin{displaymath}
\begin{tabular}{|c| c | c | c|c|}
\hline
\diagbox{$\mathbf{X}_\wedge$}{$\mathbf{Y}_\lor$}& $\mathbf{SK}_\lor$ & $\mathbf{WK}_\lor$ & $\mathbf{LMK}_\lor$ & $\mathbf{RMK}_\lor$ \\
\hline
$\mathbf{SK}_\wedge$ & \text{\cmark} & \text{\cmark} & \text{\cmark} & \text{\cmark}\\
\hline
$\mathbf{WK}_\wedge$ & \text{\cmark} & \text{\cmark} & \text{\xmark} & \text{\xmark}\\
\hline
$\mathbf{LMK}_\wedge$ & \text{\cmark} & \text{\xmark} & \text{\xmark} & \text{\xmark}\\
\hline
$\mathbf{RMK}_\wedge$ & \text{\cmark} & \text{\xmark} & \text{\xmark} & \text{\xmark}\\
\hline
\end{tabular}
\end{displaymath}
\caption{Schemes $\mathbf{X}_{\wedge}/\mathbf{Y}_{\lor}$ for which the refined property fails or holds}\label{tbl:2}
\end{center}
\end{figure}

\begin{Theorem}\label{thm:negssst}
If a scheme $\mathbf{X}$ is in $\mathcal{S}$, then it does not satisfy the following split interpolation property: 

\vspace{7pt}
If $\mathit{At}(\phi) \cap \mathit{At}(\psi) \neq \emptyset$ and $\models_{\mathbf{X}}^{\mathbf{st}} \phi \Rightarrow \psi$, then there is a formula $\chi$ such that each of its atoms is in $\phi$ and $\psi$, and $\models_{\mathbf{X}}^{\mathbf{ss}} \phi \Rightarrow \chi$ and $\models_{\mathbf{X}}^{\mathbf{st}} \chi \Rightarrow \psi$.
\end{Theorem}
\begin{proof}
We must prove that for every scheme $\mathbf{X}$ in $\mathcal{S}$, there are $\phi, \psi$ such that $\models_{\mathbf{X}}^{\mathbf{st}} \phi \Rightarrow \psi$, but there is no $\chi$ such that $\mathit{At}(\chi) \subseteq \mathit{At}(\phi) \cap \mathit{At}(\psi)$, $\models_{\mathbf{X}}^{\mathbf{ss}} \phi \Rightarrow \chi$ and $\models_{\mathbf{X}}^{\mathbf{st}} \chi \Rightarrow \psi$. For every scheme $\mathbf{X} \in \mathcal{S}$, we define $\psi = p \lor q$ and

\begin{itemize}
\setlength\itemsep{1em}
    \item $\phi = (q \wedge r) \lor (p \wedge \neg r)$, if $\mathbf{X} = \mathbf{RMK}_{\wedge}/\mathbf{Y}_{\lor}$ for some $\mathbf{Y}_{\lor}$,
    \item $\phi = (r \wedge q) \lor (\neg r \wedge p)$, if $\mathbf{X} = \mathbf{LMK}_{\wedge}/\mathbf{Y}_{\lor}$ for some $\mathbf{Y}_{\lor}$,
    \item $\phi = (r \lor q) \wedge (\neg r \lor p)$, if $\mathbf{X} = \mathbf{X}_{\wedge}/\mathbf{RMK}_{\lor}$ for some $\mathbf{X}_{\wedge}$,
    \item $\phi = (q \lor r) \wedge (p \lor \neg r)$, if $\mathbf{X} = \mathbf{X}_{\wedge}/\mathbf{LMK}_{\lor}$ for some $\mathbf{X}_{\wedge}$.
\end{itemize}
The remainder of the proof follows the same steps as in Theorem \ref{thm:nett}.
\end{proof}

\begin{Corollary}\label{cor:negstss}
    If a scheme $\mathbf{X}$ is in $\mathcal{S}$, then it does not satisfy the following split interpolation property: 

\vspace{7pt}
If $\mathit{At}(\phi) \cap \mathit{At}(\psi) \neq \emptyset$ and $\models_{\mathbf{X}}^{\mathbf{st}} \phi \Rightarrow \psi$, then there is a formula $\chi$ such that each of its atoms is in $\phi$ and $\psi$, and $\models_{\mathbf{X}}^{\mathbf{ss}} \phi \Rightarrow \chi$ and $\models_{\mathbf{X}}^{\mathbf{tt}} \chi \Rightarrow \psi$.
\end{Corollary}

\begin{proof}
    By Fact \ref{fct:subst}, for any $\mathbf{X}$, $\models_{\mathbf{X}}^{\mathbf{tt}} \ \subseteq \ \models_{\mathbf{X}}^{\mathbf{st}}$, so any counterexample to the $\mathbf{st}$-validity of an inference is also a counterexample to the $\mathbf{tt}$-validity of an inference.
\end{proof}

\begin{figure}[ht]
\begin{center}
\begin{displaymath}
\begin{tabular}{| c | c | c | c |  c | c |}
\hline
 \diagbox{$\mathcal{C}$}{$\mathcal{C'}$} & $\mathbf{ss}$ & $\mathbf{tt}$ & $\mathbf{st}$ & $\mathbf{ts}$ & $\mathbf{ss} \cap \mathbf{tt}$\\
\hline
$\mathbf{ss}$ & \text{\xmark} & \text{\cmark}/\text{\xmark}  & \text{\cmark}/\text{\xmark} & \text{\xmark} & \text{\xmark}\\
 \hline
$\mathbf{tt}$ & \text{\xmark} & \text{\xmark} & \text{\xmark} & \text{\xmark} & \text{\xmark}\\
 \hline
$\mathbf{st}$ & \text{\xmark} & \text{\cmark}/\text{\xmark} & \text{\cmark} & \text{\xmark} & \text{\xmark}\\
 \hline
$\mathbf{ts}$ & \text{\xmark} & \text{\xmark} & \text{\xmark} & \text{\xmark} & \text{\xmark}\\
 \hline
$\mathbf{ss} \cap \mathbf{tt}$ & \text{\xmark} & \text{\xmark} & \text{\xmark} & \text{\xmark} & \text{\xmark}\\
\hline
\end{tabular}
\end{displaymath}
\caption{Failure/success of the property}\label{tbl:3}
\end{center}
\end{figure}

By fixing the combination of an $\mathbf{ss}$-consequence relation with a $\mathbf{tt}$-consequence relation, and incorporating these negative results with the earlier positive ones, we obtain the outcomes shown in Figure \ref{tbl:2}. In light of Corollary \ref{cor:posst+}, along with theorems \ref{thm:nett} and \ref{thm:negssst}, this table also extends to the $\mathbf{st}/\mathbf{tt}$ and $\mathbf{ss}/\mathbf{st}$ combinations. This now allows us to complete Figure \ref{tbl:1}, resulting in the formation of Figure \ref{tbl:3}. The combination of symbols $\text{\cmark}/\text{\xmark}$ indicates that, for the corresponding pair of standards, the property fails or holds depending on the scheme selected, and directs the reader to refer to Figure \ref{tbl:2}.

\section{Discussion}\label{sec:dis}
There are 400 possible combinations of three-valued monotonic logics based on the same Boolean normal monotonic scheme, yet, as we have seen, only 40 of them satisfy the mixed interpolation property. Among these 40 schemes, some appear more refined than others. For example, the version of the split interpolation based on a pair of consequence relations $\models_{\mathbf{WK}_\wedge/\mathbf{WK}_\lor}^{\mathbf{ss}}$ and $\models_{\mathbf{WK}_\wedge/\mathbf{WK}_\lor}^{\mathbf{tt}}$ entails the split interpolation based on the pair of consequence relations $\models_{\mathbf{WK}_\wedge/\mathbf{WK}_\lor}^{\mathbf{st}}$ and $\models_{\mathbf{WK}_\wedge/\mathbf{WK}_\lor}^{\mathbf{tt}}$.
By Fact \ref{fct:subst}, $\models_{\mathbf{X}}^{\mathbf{ss}} \subseteq \models_{\mathbf{X}}^{\mathbf{st}}$, so the existence of an interpolant in the first case guarantees the existence of one in the second.  In general, a refinement will be said to be stronger than another if the former entails the latter. Any split interpolation result involving a consequence relation $\models_{\mathbf{X}}^{\mathbf{st}}$ can be reduced to another involving a consequence relation $\models_{\mathbf{X}}^{\mathbf{ss}}$ and a consequence relation $\models_{\mathbf{X}}^{\mathbf{tt}}$, depending on the cases. Consequently, any split interpolation property involving a pair of logics based on an $\mathbf{ss}$- and a $\mathbf{tt}$-consequence relation, respectively, is more refined than any other. In such a case, the choice of scheme for these two consequence relations is crucial in determining whether the property holds or not. As we have seen, the split interpolation property holds for logics with standards $\mathbf{ss}$ and $\mathbf{tt}$ only when the scheme includes either a Strong Kleene conjunction or a Strong Kleene disjunction, or both a Weak Kleene conjunction and a Weak Kleene disjunction. We found that the property holds in eight cases and fails in the remaining eight. None of the cases where the property holds can be reduced to another, as the underlying consequence relations are not subsets of one another. They therefore constitute our eight refinements of Craig's interpolation theorem.

In the formulation of the split interpolation property, one effect of using a paracomplete logic on the left---such as any logic defined by a $\mathbf{ss}$-relation with a Boolean normal monotonic scheme---and a paraconsistent logic on the right---such as any logic with a $\mathbf{tt}$-relation similarly combined---is a reduction in the possible options for the interpolant. Excluding trivial cases, choosing an interpolant that is either a contradiction or a tautology will systematically fail, even if it is entailed by the premises or entails the conclusion.\footnote{Given a classically valid inference $\phi \Rightarrow \psi$, the assumption of the interpolation property requires that $\mathit{At}(\phi) \cap \mathit{At}(\psi) \neq \emptyset$. This implies that $\phi$ is not a contradiction and $\psi$ is not a tautology---except in trivial cases, such as when $\phi = p \wedge \neg p$ and $\psi = p$, where $p$ follows from $p \wedge \neg p$ by conjunction elimination.} In the case of the split interpolation considered here, this limitation is directly embedded in the logics used: the paracomplete logic prevents the interpolant from being a tautology, while the paraconsistent logic prevents it from being a contradiction. This restriction further extends to interpolants containing contradictions or tautologies as subformulas. For instance, the inference from $p \lor (q \wedge \neg q)$ to $p \lor q$ is classically valid, and both $p \wedge (q \lor \neg q)$ and $p$ are suitable interpolants in the classical sense of Craig's interpolation. However, $p \wedge (q \lor \neg q)$ is not a suitable interpolant for the split interpolation combining an $\mathbf{ss}$-consequence relation and a $\mathbf{tt}$-consequence relation based on the Strong Kleene scheme---unlike $p$---since $\not\models_{\mathbf{SK}_{\wedge}/\mathbf{SK}_{\lor}}^{\mathbf{ss}} p \lor (q \wedge \neg q) \Rightarrow p \wedge (q \lor \neg q)$. Having an interpolant with a single variable is preferable, as it simplifies the logical structure and reduces complexity compared to an interpolant involving multiple variables.

Can these results be refined further? One approach could be to consider logics from which those described in this paper are proper extensions, meaning their sets of validities are proper subsets of the validities of the logics analyzed here. In the case of logics based on the Strong Kleene scheme, \cite{Prenosil2017} demonstrates that the split interpolation property holds for the pair of logics $\mathsf{ETL}$ and $\mathsf{LP}$. $\mathsf{ETL}$, a four-valued paracomplete logic explored in \cite{PietzRivieccio2013}, is extended by $\mathsf{K}_3$. The split interpolation involving $\mathsf{LP}$ and $\mathsf{ETL}$ is therefore more refined than Milne's interpolation. However, the refinement could potentially be taken even further. Just as $\mathsf{K}_3$ is known to be the dual of $\mathsf{LP}$ (see \citealp{blomet2024sttsproductsum}), $\mathsf{ETL}$ is the dual of $\mathsf{NFL}$, a four-valued paraconsistent logic extended by $\mathsf{LP}$ \citep{Shramko_2019}. This raises the question: does the split interpolation hold for the pair $\mathsf{ETL}$ and $\mathsf{NFL}$? Furthermore, both logics are based on a scheme that extends the Strong Kleene scheme, meaning that any trivaluation on this scheme is also a trivaluation on the Strong Kleene scheme. On that account, is it possible to extend other Boolean normal monotonic schemes to derive, for example, the Weak Kleene counterparts of $\mathsf{ETL}$ and $\mathsf{NFL}$, and if so, do they have the split interpolation property? We would arrive at results fully analogous to those presented here, with an even more refined interpolation theorem. We leave the exploration of these questions for future work.

\section{Conclusion}
We have presented results that strengthen Craig's deductive interpolation theorem for classical propositional logic. Our findings build on the refinements previously established by \cite{milne2016refinement} and \cite{Prenosil2017}, which were limited to the pair $\mathsf{K}_3$ and $\mathsf{LP}$. Out of the 400 possible combinations of two three-valued monotonic logics based on the same Boolean normal monotonic scheme, 40 satisfy the mixed interpolation property, with eight representing stronger refinements of Craig’s interpolation theorem.

Our proof was developed systematically. We began by defining the key concepts of monotonicity and Boolean normality, followed by addressing the main question for the pair of logics where the interpolation property fails, regardless of the specific connective scheme. We then explored the results that depend on the choice of scheme, and concluded by discussing the broader implications of our findings. Finally, we pointed to potential directions for future research, opening the door to further refinements of the split interpolation results.

\section*{Acknowledgements}
{\small We thank Eduardo Barrio, Agustina Borzi, Paul {\'Egr\'e}, Bruno Da R\'e, Camillo Fiore, Camila Gallovich, Federico Pailos, Dave Ripley, Mariela Rubin, Damian Szmuc, and Allard Tamminga for helpful exchanges related to the topic of this article. We acknowledge the audience of the XII Workshop on Philosophical Logic, held in Buenos Aires in August 2023; the audience of PALLMYR XIII, held in London in November 2023; and the audience of the Workshop on Theory and Applications of Craig Interpolation and Beth Definability, held in Amsterdam in April 2024.}

\section*{Funding}
{\small This project has received support from the program ECOS-SUD (``Logical consequence and many-valued models'', no.\ A22H01), a bilateral exchange program between Argentina and France, and from PLEXUS (``Philosophical, Logical and Experimental Perspectives on Substructural Logics'', grant agreement no. 101086295), a Marie Sk\l odowska-Curie action funded by the EU under the Horizon Europe Research and Innovation Program.}

\bibliographystyle{apacite}
\bibliography{references}

\begin{thebibliography}{}

\bibitem [\protect \citeauthoryear {%
Asenjo%
}{%
Asenjo%
}{%
{\protect \APACyear {1966}}%
}]{%
Asenjo1966}
\APACinsertmetastar {%
Asenjo1966}%
\begin{APACrefauthors}%
Asenjo, F.%
\end{APACrefauthors}%
\unskip\
\newblock
\APACrefYearMonthDay{1966}{}{}.
\newblock
{\BBOQ}\APACrefatitle {A calculus of antinomies} {A calculus of antinomies}.{\BBCQ}
\newblock
\APACjournalVolNumPages{Notre Dame Journal of Formal Logic}{15}{}{497--509}.
\PrintBackRefs{\CurrentBib}

\bibitem [\protect \citeauthoryear {%
Beaver%
\ \BBA {} Krahmer%
}{%
Beaver%
\ \BBA {} Krahmer%
}{%
{\protect \APACyear {2001}}%
}]{%
BeaverKrahmer2001}
\APACinsertmetastar {%
BeaverKrahmer2001}%
\begin{APACrefauthors}%
Beaver, D.%
\BCBT {}\ \BBA {} Krahmer, E.%
\end{APACrefauthors}%
\unskip\
\newblock
\APACrefYearMonthDay{2001}{}{}.
\newblock
{\BBOQ}\APACrefatitle {A partial account of presupposition projection} {A partial account of presupposition projection}.{\BBCQ}
\newblock
\APACjournalVolNumPages{Journal of Logic, Language and Information}{10}{2}{147--182}.
\PrintBackRefs{\CurrentBib}

\bibitem [\protect \citeauthoryear {%
Blomet%
\ \BBA {} \'Egr\'e%
}{%
Blomet%
\ \BBA {} \'Egr\'e%
}{%
{\protect \APACyear {2024}}%
}]{%
blomet2024sttsproductsum}
\APACinsertmetastar {%
blomet2024sttsproductsum}%
\begin{APACrefauthors}%
Blomet, Q.%
\BCBT {}\ \BBA {} \'Egr\'e, P.%
\end{APACrefauthors}%
\unskip\
\newblock
\APACrefYearMonthDay{2024}{}{}.
\newblock
{\BBOQ}\APACrefatitle {{ST} and {TS} as product and sum} {{ST} and {TS} as product and sum}.{\BBCQ}
\newblock
\APACjournalVolNumPages{Journal of Philosophical Logic}{53}{6}{1673–-1700}.
\PrintBackRefs{\CurrentBib}

\bibitem [\protect \citeauthoryear {%
Bochvar%
}{%
Bochvar%
}{%
{\protect \APACyear {1938}}%
}]{%
Bochvar1938}
\APACinsertmetastar {%
Bochvar1938}%
\begin{APACrefauthors}%
Bochvar, D.%
\end{APACrefauthors}%
\unskip\
\newblock
\APACrefYearMonthDay{1938}{}{}.
\newblock
{\BBOQ}\APACrefatitle {On a three-valued calculus and its application in the analysis of the paradoxes of the extended functional calculus} {On a three-valued calculus and its application in the analysis of the paradoxes of the extended functional calculus}.{\BBCQ}
\newblock
\APACjournalVolNumPages{Matamaticheskii Sbornik}{4}{}{287--308}.
\PrintBackRefs{\CurrentBib}

\bibitem [\protect \citeauthoryear {%
Chemla%
, \'Egr\'e%
\BCBL {}\ \BBA {} Spector%
}{%
Chemla%
\ \protect \BOthers {.}}{%
{\protect \APACyear {2017}}%
}]{%
Chemlaetal2017}
\APACinsertmetastar {%
Chemlaetal2017}%
\begin{APACrefauthors}%
Chemla, E.%
, \'Egr\'e, P.%
\BCBL {}\ \BBA {} Spector, B.%
\end{APACrefauthors}%
\unskip\
\newblock
\APACrefYearMonthDay{2017}{}{}.
\newblock
{\BBOQ}\APACrefatitle {Characterizing logical consequence in many-valued logic} {Characterizing logical consequence in many-valued logic}.{\BBCQ}
\newblock
\APACjournalVolNumPages{Journal of Logic and Computation}{27}{7}{2193-2226}.
\PrintBackRefs{\CurrentBib}

\bibitem [\protect \citeauthoryear {%
Cobreros%
, \'Egr{\'e}%
, Ripley%
\BCBL {}\ \BBA {} van Rooij%
}{%
Cobreros%
\ \protect \BOthers {.}}{%
{\protect \APACyear {2012}}%
}]{%
Cobrerosetal2012}
\APACinsertmetastar {%
Cobrerosetal2012}%
\begin{APACrefauthors}%
Cobreros, P.%
, \'Egr{\'e}, P.%
, Ripley, D.%
\BCBL {}\ \BBA {} van Rooij, R.%
\end{APACrefauthors}%
\unskip\
\newblock
\APACrefYearMonthDay{2012}{}{}.
\newblock
{\BBOQ}\APACrefatitle {Tolerant, classical, strict} {Tolerant, classical, strict}.{\BBCQ}
\newblock
\APACjournalVolNumPages{Journal of Philosophical Logic}{41}{2}{347-385}.
\PrintBackRefs{\CurrentBib}

\bibitem [\protect \citeauthoryear {%
Craig%
}{%
Craig%
}{%
{\protect \APACyear {1957}}%
}]{%
Craig1957}
\APACinsertmetastar {%
Craig1957}%
\begin{APACrefauthors}%
Craig, W.%
\end{APACrefauthors}%
\unskip\
\newblock
\APACrefYearMonthDay{1957}{}{}.
\newblock
{\BBOQ}\APACrefatitle {Linear reasoning. A new form of the {H}erbrand-{G}entzen theorem} {Linear reasoning. a new form of the {H}erbrand-{G}entzen theorem}.{\BBCQ}
\newblock
\APACjournalVolNumPages{Journal of Symbolic Logic}{22}{}{250--268}.
\PrintBackRefs{\CurrentBib}

\bibitem [\protect \citeauthoryear {%
{Da R{\'e}}%
, Szmuc%
, Chemla%
\BCBL {}\ \BBA {} {\'E}gr{\'e}%
}{%
{Da R{\'e}}%
\ \protect \BOthers {.}}{%
{\protect \APACyear {2023}}%
}]{%
da2023three}
\APACinsertmetastar {%
da2023three}%
\begin{APACrefauthors}%
{Da R{\'e}}, B.%
, Szmuc, D.%
, Chemla, E.%
\BCBL {}\ \BBA {} {\'E}gr{\'e}, P.%
\end{APACrefauthors}%
\unskip\
\newblock
\APACrefYearMonthDay{2023}{}{}.
\newblock
{\BBOQ}\APACrefatitle {On three-valued presentations of classical logic} {On three-valued presentations of classical logic}.{\BBCQ}
\newblock
\APACjournalVolNumPages{The Review of Symbolic Logic}{}{}{1--23}.
\PrintBackRefs{\CurrentBib}

\bibitem [\protect \citeauthoryear {%
{Di Nola}%
\ \BBA {} Lettieri%
}{%
{Di Nola}%
\ \BBA {} Lettieri%
}{%
{\protect \APACyear {2000}}%
}]{%
NolaLettieri2000}
\APACinsertmetastar {%
NolaLettieri2000}%
\begin{APACrefauthors}%
{Di Nola}, A.%
\BCBT {}\ \BBA {} Lettieri, A.%
\end{APACrefauthors}%
\unskip\
\newblock
\APACrefYearMonthDay{2000}{}{}.
\newblock
{\BBOQ}\APACrefatitle {One chain generated varieties of {MV}-algebras} {One chain generated varieties of {MV}-algebras}.{\BBCQ}
\newblock
\APACjournalVolNumPages{Journal of Algebra}{225}{2}{667-–697}.
\PrintBackRefs{\CurrentBib}

\bibitem [\protect \citeauthoryear {%
Field%
}{%
Field%
}{%
{\protect \APACyear {2008}}%
}]{%
Field2008}
\APACinsertmetastar {%
Field2008}%
\begin{APACrefauthors}%
Field, H.%
\end{APACrefauthors}%
\unskip\
\newblock
\APACrefYear{2008}.
\newblock
\APACrefbtitle {Saving Truth from Paradox} {Saving truth from paradox}.
\newblock
\APACaddressPublisher{New York}{Oxford University Press}.
\PrintBackRefs{\CurrentBib}

\bibitem [\protect \citeauthoryear {%
Frankowski%
}{%
Frankowski%
}{%
{\protect \APACyear {2004}}%
}]{%
Frankowski2004}
\APACinsertmetastar {%
Frankowski2004}%
\begin{APACrefauthors}%
Frankowski, S.%
\end{APACrefauthors}%
\unskip\
\newblock
\APACrefYearMonthDay{2004}{}{}.
\newblock
{\BBOQ}\APACrefatitle {Formalization of a plausible inference} {Formalization of a plausible inference}.{\BBCQ}
\newblock
\APACjournalVolNumPages{Bulletin of the Section of Logic}{33}{1}{41--52}.
\PrintBackRefs{\CurrentBib}

\bibitem [\protect \citeauthoryear {%
George%
}{%
George%
}{%
{\protect \APACyear {2014}}%
}]{%
George2014}
\APACinsertmetastar {%
George2014}%
\begin{APACrefauthors}%
George, B.%
\end{APACrefauthors}%
\unskip\
\newblock
\APACrefYearMonthDay{2014}{}{}.
\newblock
{\BBOQ}\APACrefatitle {Some remarks on certain trivalent accounts of presupposition projection} {Some remarks on certain trivalent accounts of presupposition projection}.{\BBCQ}
\newblock
\APACjournalVolNumPages{Journal of Applied Non-classical Logics}{24}{1-2}{86--117}.
\PrintBackRefs{\CurrentBib}

\bibitem [\protect \citeauthoryear {%
Girard%
}{%
Girard%
}{%
{\protect \APACyear {1987}}%
}]{%
Girard1987}
\APACinsertmetastar {%
Girard1987}%
\begin{APACrefauthors}%
Girard, J\BHBI Y.%
\end{APACrefauthors}%
\unskip\
\newblock
\APACrefYearMonthDay{1987}{}{}.
\newblock
{\BBOQ}\APACrefatitle {Linear logic} {Linear logic}.{\BBCQ}
\newblock
\APACjournalVolNumPages{Theoretical Computer Science}{50}{1}{1--101}.
\PrintBackRefs{\CurrentBib}

\bibitem [\protect \citeauthoryear {%
Halld\'en%
}{%
Halld\'en%
}{%
{\protect \APACyear {1949}}%
}]{%
Hallden1949}
\APACinsertmetastar {%
Hallden1949}%
\begin{APACrefauthors}%
Halld\'en, S.%
\end{APACrefauthors}%
\unskip\
\newblock
\APACrefYear{1949}.
\unskip\
\newblock
\APACrefbtitle {The Logic of Nonsense} {The logic of nonsense}\ \APACtypeAddressSchool {\BUPhD}{}{}.
\unskip\
\newblock
\APACaddressSchool {}{Uppsala University}.
\PrintBackRefs{\CurrentBib}

\bibitem [\protect \citeauthoryear {%
Hintikka%
\ \BBA {} Halonen%
}{%
Hintikka%
\ \BBA {} Halonen%
}{%
{\protect \APACyear {1999}}%
}]{%
Hintikka1999}
\APACinsertmetastar {%
Hintikka1999}%
\begin{APACrefauthors}%
Hintikka, J.%
\BCBT {}\ \BBA {} Halonen, I.%
\end{APACrefauthors}%
\unskip\
\newblock
\APACrefYearMonthDay{1999}{}{}.
\newblock
{\BBOQ}\APACrefatitle {Interpolation as explanation} {Interpolation as explanation}.{\BBCQ}
\newblock
\BIn{} \APACrefbtitle {Proceedings of the 1998 Biennial Meetings of the Philosophy of Science Association. Part I: Contributed Papers (Sep., 1999)} {Proceedings of the 1998 biennial meetings of the philosophy of science association. part i: Contributed papers (sep., 1999)}\ (\BVOLS\ 66, Supplement, \BPGS\ 414--423).
\newblock
\APACaddressPublisher{Chicago}{The University of Chicago Press}.
\PrintBackRefs{\CurrentBib}

\bibitem [\protect \citeauthoryear {%
Kleene%
}{%
Kleene%
}{%
{\protect \APACyear {1952}}%
}]{%
Kleene1952-KLEITM}
\APACinsertmetastar {%
Kleene1952-KLEITM}%
\begin{APACrefauthors}%
Kleene, S.%
\end{APACrefauthors}%
\unskip\
\newblock
\APACrefYear{1952}.
\newblock
\APACrefbtitle {Introduction to Metamathematics} {Introduction to metamathematics}.
\newblock
\APACaddressPublisher{Groningen}{P. Noordhoff N.V.}
\PrintBackRefs{\CurrentBib}

\bibitem [\protect \citeauthoryear {%
Makinson%
}{%
Makinson%
}{%
{\protect \APACyear {1973}}%
}]{%
Makinson1973}
\APACinsertmetastar {%
Makinson1973}%
\begin{APACrefauthors}%
Makinson, D.%
\end{APACrefauthors}%
\unskip\
\newblock
\APACrefYear{1973}.
\newblock
\APACrefbtitle {Topics in Modern Logic} {Topics in modern logic}.
\newblock
\APACaddressPublisher{London}{Methuen}.
\PrintBackRefs{\CurrentBib}

\bibitem [\protect \citeauthoryear {%
Maksimova%
}{%
Maksimova%
}{%
{\protect \APACyear {1977}}%
}]{%
Maksimova1977}
\APACinsertmetastar {%
Maksimova1977}%
\begin{APACrefauthors}%
Maksimova, L.%
\end{APACrefauthors}%
\unskip\
\newblock
\APACrefYearMonthDay{1977}{}{}.
\newblock
{\BBOQ}\APACrefatitle {Craig's theorem in superintuitionistic logics and amalgamable varieties of pseudo-boolean algebras} {Craig's theorem in superintuitionistic logics and amalgamable varieties of pseudo-boolean algebras}.{\BBCQ}
\newblock
\APACjournalVolNumPages{Algebra and Logic}{16}{}{427--455}.
\PrintBackRefs{\CurrentBib}

\bibitem [\protect \citeauthoryear {%
Malinowski%
}{%
Malinowski%
}{%
{\protect \APACyear {1990}}%
}]{%
Malinowski1990}
\APACinsertmetastar {%
Malinowski1990}%
\begin{APACrefauthors}%
Malinowski, G.%
\end{APACrefauthors}%
\unskip\
\newblock
\APACrefYearMonthDay{1990}{}{}.
\newblock
{\BBOQ}\APACrefatitle {Q-consequence operation} {Q-consequence operation}.{\BBCQ}
\newblock
\APACjournalVolNumPages{Reports on Mathematical Logic}{24}{}{49-54}.
\PrintBackRefs{\CurrentBib}

\bibitem [\protect \citeauthoryear {%
McMillan%
}{%
McMillan%
}{%
{\protect \APACyear {2003}}%
}]{%
McMillan2003}
\APACinsertmetastar {%
McMillan2003}%
\begin{APACrefauthors}%
McMillan, K\BPBI L.%
\end{APACrefauthors}%
\unskip\
\newblock
\APACrefYearMonthDay{2003}{}{}.
\newblock
{\BBOQ}\APACrefatitle {Interpolation and {SAT}-based model checking} {Interpolation and {SAT}-based model checking}.{\BBCQ}
\newblock
\BIn{} W\BPBI A.~Hunt\ \BBA {} F.~Somenzi\ (\BEDS), \APACrefbtitle {Computer Aided Verification} {Computer aided verification}\ (\BPGS\ 1--13).
\newblock
\APACaddressPublisher{Berlin}{Springer}.
\PrintBackRefs{\CurrentBib}

\bibitem [\protect \citeauthoryear {%
Milne%
}{%
Milne%
}{%
{\protect \APACyear {2016}}%
}]{%
milne2016refinement}
\APACinsertmetastar {%
milne2016refinement}%
\begin{APACrefauthors}%
Milne, P.%
\end{APACrefauthors}%
\unskip\
\newblock
\APACrefYearMonthDay{2016}{}{}.
\newblock
{\BBOQ}\APACrefatitle {A non-classical refinement of the interpolation property for classical propositional logic} {A non-classical refinement of the interpolation property for classical propositional logic}.{\BBCQ}
\newblock
\APACjournalVolNumPages{Logique et Analyse}{59}{235}{273--281}.
\PrintBackRefs{\CurrentBib}

\bibitem [\protect \citeauthoryear {%
Peters%
}{%
Peters%
}{%
{\protect \APACyear {1979}}%
}]{%
Peters1979}
\APACinsertmetastar {%
Peters1979}%
\begin{APACrefauthors}%
Peters, S.%
\end{APACrefauthors}%
\unskip\
\newblock
\APACrefYearMonthDay{1979}{}{}.
\newblock
{\BBOQ}\APACrefatitle {A truth-conditional formulation of Karttunen' saccount of presupposition} {A truth-conditional formulation of karttunen' saccount of presupposition}.{\BBCQ}
\newblock
\APACjournalVolNumPages{Synthese}{40}{2}{301--316}.
\PrintBackRefs{\CurrentBib}

\bibitem [\protect \citeauthoryear {%
Pietz%
\ \BBA {} Rivieccio%
}{%
Pietz%
\ \BBA {} Rivieccio%
}{%
{\protect \APACyear {2013}}%
}]{%
PietzRivieccio2013}
\APACinsertmetastar {%
PietzRivieccio2013}%
\begin{APACrefauthors}%
Pietz, A.%
\BCBT {}\ \BBA {} Rivieccio, U.%
\end{APACrefauthors}%
\unskip\
\newblock
\APACrefYearMonthDay{2013}{}{}.
\newblock
{\BBOQ}\APACrefatitle {Nothing but the truth} {Nothing but the truth}.{\BBCQ}
\newblock
\APACjournalVolNumPages{Journal of Philosophical Logic}{42}{1}{125--135}.
\PrintBackRefs{\CurrentBib}

\bibitem [\protect \citeauthoryear {%
Priest%
}{%
Priest%
}{%
{\protect \APACyear {1979}}%
}]{%
Priest1979}
\APACinsertmetastar {%
Priest1979}%
\begin{APACrefauthors}%
Priest, G.%
\end{APACrefauthors}%
\unskip\
\newblock
\APACrefYearMonthDay{1979}{}{}.
\newblock
{\BBOQ}\APACrefatitle {The logic of paradox} {The logic of paradox}.{\BBCQ}
\newblock
\APACjournalVolNumPages{Journal of Philosophical Logic}{8}{1}{219-241}.
\PrintBackRefs{\CurrentBib}

\bibitem [\protect \citeauthoryear {%
P\v{r}enosil%
}{%
P\v{r}enosil%
}{%
{\protect \APACyear {2017}}%
}]{%
Prenosil2017}
\APACinsertmetastar {%
Prenosil2017}%
\begin{APACrefauthors}%
P\v{r}enosil, A.%
\end{APACrefauthors}%
\unskip\
\newblock
\APACrefYearMonthDay{2017}{}{}.
\newblock
{\BBOQ}\APACrefatitle {Cut elimination, identity elimination, and interpolation in super-{B}elnap logics} {Cut elimination, identity elimination, and interpolation in super-{B}elnap logics}.{\BBCQ}
\newblock
\APACjournalVolNumPages{Studia Logica}{105}{6}{1255--1289}.
\PrintBackRefs{\CurrentBib}

\bibitem [\protect \citeauthoryear {%
Scambler%
}{%
Scambler%
}{%
{\protect \APACyear {2020}}%
}]{%
scambler2020classical}
\APACinsertmetastar {%
scambler2020classical}%
\begin{APACrefauthors}%
Scambler, C.%
\end{APACrefauthors}%
\unskip\
\newblock
\APACrefYearMonthDay{2020}{}{}.
\newblock
{\BBOQ}\APACrefatitle {Classical logic and the strict tolerant hierarchy} {Classical logic and the strict tolerant hierarchy}.{\BBCQ}
\newblock
\APACjournalVolNumPages{Journal of Philosophical Logic}{49}{2}{351--370}.
\PrintBackRefs{\CurrentBib}

\bibitem [\protect \citeauthoryear {%
Sch{\"u}tte%
}{%
Sch{\"u}tte%
}{%
{\protect \APACyear {1962}}%
}]{%
Schutte1962}
\APACinsertmetastar {%
Schutte1962}%
\begin{APACrefauthors}%
Sch{\"u}tte, K.%
\end{APACrefauthors}%
\unskip\
\newblock
\APACrefYearMonthDay{1962}{}{}.
\newblock
{\BBOQ}\APACrefatitle {Der {I}nterpolationssatz der intuitionistischen {P}rädikatenlogik} {Der {I}nterpolationssatz der intuitionistischen {P}rädikatenlogik}.{\BBCQ}
\newblock
\APACjournalVolNumPages{Mathematische Annalen}{39}{}{192--200}.
\PrintBackRefs{\CurrentBib}

\bibitem [\protect \citeauthoryear {%
Shramko%
}{%
Shramko%
}{%
{\protect \APACyear {2019}}%
}]{%
Shramko_2019}
\APACinsertmetastar {%
Shramko_2019}%
\begin{APACrefauthors}%
Shramko, Y.%
\end{APACrefauthors}%
\unskip\
\newblock
\APACrefYearMonthDay{2019}{}{}.
\newblock
{\BBOQ}\APACrefatitle {Dual-Belnap logic and anything but falsehood} {Dual-belnap logic and anything but falsehood}.{\BBCQ}
\newblock
\APACjournalVolNumPages{Journal of Applied Logics – IfCoLoG Journal of Logics and their Applications}{6}{2}{413-430}.
\PrintBackRefs{\CurrentBib}

\end{thebibliography}

\end{document}